\pdfoutput=1
\documentclass{article}
\usepackage{iclr2027_conference,times}

\usepackage{amsmath,amssymb,amsthm}
\usepackage{bm}
\usepackage{booktabs}
\usepackage{graphicx}
\usepackage{hyperref}
\usepackage{url}

\theoremstyle{plain}
\newtheorem{theorem}{Theorem}
\newtheorem{proposition}[theorem]{Proposition}
\newtheorem{lemma}[theorem]{Lemma}
\newtheorem{corollary}[theorem]{Corollary}
\newtheorem{assumption}{Assumption}
\theoremstyle{remark}

\newcommand{\E}{\mathbb{E}}
\newcommand{\Prob}{\mathbb{P}}
\newcommand{\R}{\mathbb{R}}
\newcommand{\Var}{\mathrm{Var}}
\newcommand{\Cov}{\mathrm{Cov}}
\newcommand{\ind}[1]{\mathbf{1}\{#1\}}
\newcommand{\len}{\operatorname{len}}
\newcommand{\norm}[1]{\|#1\|}
\newcommand{\abs}[1]{|#1|}
\newcommand{\iid}{\text{i.i.d.}}
\newcommand{\SQ}{S^{Q}}
\newcommand{\SM}{S^{M}}
\newcommand{\SQs}{S^{Q*}}
\newcommand{\SMs}{S^{M*}}
\newcommand{\ShatQ}{\widehat{S}^{Q}}
\newcommand{\ShatM}{\widehat{S}^{M}}
\newcommand{\CQ}{C^{Q}_{\alpha}}
\newcommand{\CM}{C^{M}_{\alpha}}
\newcommand{\ChatQ}{\widehat{C}^{Q}_{\alpha}}
\newcommand{\ChatM}{\widehat{C}^{M}_{\alpha}}
\newcommand{\EQ}{E^{Q}_{n}}
\newcommand{\EM}{E^{M}_{n}}
\newcommand{\Sor}{S}
\newcommand{\Qor}{S_{(k)}}

\newcommand{\Shat}{\widehat{S}}

\newcommand{\lhat}{\widehat{\ell}}
\newcommand{\uhat}{\widehat{u}}
\newcommand{\muhat}{\widehat{\mu}}

\newcommand{\qR}{\ShatM_{(k)}}
\newcommand{\qRor}{\SM_{(k)}}
\newcommand{\Fhat}{\widehat{F}_m}
\newcommand{\Hhat}{\widehat{H}_m}

\newcommand{\Gbar}{\bar{G}_n}

\newcommand{\calP}{\mathcal{P}}
\newcommand{\Fcal}{\mathcal{F}}
\newcommand{\tr}{\top}

\title{Valid and Efficient Split Conformal\\ Regression for Time Series}

\author{Percy S. Zhai\textsuperscript{1}, Maggie Cheng\textsuperscript{2}, Wei Biao Wu\textsuperscript{1} \\
\textsuperscript{1}University of Chicago \quad \textsuperscript{2}Illinois Institute of Technology \\
\small\texttt{percy.zhai@chicagobooth.edu}, \texttt{maggie.cheng@iit.edu}, \texttt{wbwu@uchicago.edu}}

\iclrfinalcopy
\begin{document}

\maketitle
\lhead{}
\renewcommand{\headrulewidth}{0pt}
\vspace{-11.70287pt}

\begin{abstract}
We study conformalized quantile regression and conformalized median regression that fit a model on one block of a time series and calibrate the conformal interval on the adjacent block. The existing theory of conformal prediction for time series rests largely on mixing conditions, which are hard to verify from a time-series model and fail for many standard processes, including simple ones with short memory. We replace this theoretical toolbox with the functional dependence measure, which in principle accommodates long-memory observations. The accuracy of the conformal interval length for time series has been understudied. To the best of our knowledge, this paper is the first work that establishes non-asymptotic coverage guarantees and accuracy of interval length simultaneously for split conformal regression on time series. Furthermore, for Gaussian linear processes with long memory, where both the estimation of the center and its calibration converge slowly, we establish a sharper rate for the length error. We show that the calibrated length converges faster than the estimated center itself, and provide a matching lower bound for the usual centers when the calibration block is sufficiently large relative to the training block. To our knowledge, this is the first theoretical analysis of conformal interval length dedicated to long memory.
\end{abstract}

\section{Introduction}\label{sec:intro}

Split conformal prediction turns a fitted regression model into a prediction interval: the model is fitted on a training block, and the interval for a new observation is the fitted interval enlarged by an empirical quantile of scores computed on a separate calibration block \citep{papadopoulos2002inductive,vovk2005algorithmic,lei2018distribution}. For conformalized quantile regression \citep[CQR;][]{romano2019conformalized} the model consists of two fitted conditional quantile functions, and for conformalized median regression (CMR) of a single fitted center whose absolute residuals are calibrated. When the calibration and test observations are exchangeable, the interval covers the new response with probability at least the nominal level whatever the fitted model is. For a time series the two blocks and the test point are adjacent segments of one trajectory, exchangeability fails, and two questions arise: whether the interval still covers at the nominal level, and whether its length approaches the length of the oracle interval formed by the true conditional quantiles. The second question is the question of efficiency, and the two questions together are the subject of this paper.

For non-exchangeable data, the coverage gap of conformal prediction can be bounded through total variation distances \citep{barber2023beyond}; for a stationary series, the gap of split, block, jackknife and ensemble procedures has been quantified through mixing conditions \citep{chernozhukov2018exact,xu2021conformal,oliveira2024split,barber2026predictive,jiang2026leave}, the standard toolbox for learning from dependent data \citep{yu1994rates,mohri2010stability}. This toolbox has two limitations: mixing coefficients are rarely computable from the model that generates a series \citep{wu2005nonlinear,wu2011asymptotic,han2023probability}, and mixing conditions fail for many simple time series \citep{andrews1984nonstrong,bradley2005basic}. Long memory, the property that the autocorrelations of a series are not summable \citep{beran1994statistics}, is documented for asset returns and network traffic \citep{ding1993long,leland1994self} and in language and music \citep{greavestunnell2019statistical}, and standard recurrent networks do not capture it \citep{zhao2020rnn}. The functional dependence measure of \citet{wu2005nonlinear} is computed from the model that generates the series and underlies a theory of quantiles and of high-dimensional time series \citep{wu2005bahadur,wu2011asymptotic,chen2013covariance,zhang2017gaussian}; for a linear process it is determined by the coefficients of the process, and the empirical processes and sample covariance matrices of linear processes are understood also under long memory \citep{wu2003empirical,chen2018concentration,zhai2026simultaneous}.

The length of conformal intervals has been studied for exchangeable data: split conformal and CQR intervals approach the oracle length when the fitted model is consistent \citep{lei2018distribution,sesia2020comparison}, and \citet{yao2026nonasymptotic} give non-asymptotic rates for the length error of CQR and CMR. For dependent data, the bound of \citet[Theorem~3]{xu2023conformal} on the discrepancy between their ensemble interval and the oracle interval is proved for \iid\ errors; the efficiency analysis of adaptive conformal inference \citep{gibbs2021adaptive} by \citet{zaffran2022adaptive} is asymptotic and assumes a known quantile function of the scores; and distributional conformal prediction, a different construction, attains an optimal interval only asymptotically \citep{chernozhukov2021distributional}. None of these compares the calibrated length with the oracle length at a non-asymptotic rate under temporal dependence. \citet{zheng2024conformal} bound the effect of Markov dependence on the size of the split conformal set, but only for Markov chains with a finite uniform mixing time, and show that the set converges to the oracle set without a rate.

\citet{sesia2021conformal}, \citet{kiyani2024length} and \citet{lebars2025volume} study efficiency and length optimization for other oracle targets; we compare with the marginal equal-tailed oracle. \citet{halkiewicz2026rolling} analyzes the coverage of rolling-origin conformal prediction under strong mixing and, in an extension, under a summability condition on the physical dependence measure; neither condition allows Gaussian long memory.

Our contributions can be summarized as follows.
\begin{enumerate}
\item The mixing conditions on which the existing guarantees largely rest fail for many standard processes, including simple ones with short memory. We replace this theoretical toolbox with the functional dependence measure, under which our guarantees hold for such processes and, in principle, allow for long memory.
\item For CQR and CMR with adjacent training and calibration blocks, we bound the coverage error of the calibrated interval and the deviation of its length from that of the oracle interval (Theorems~\ref{thm:coverage}--\ref{thm:cmr}). To our knowledge, these are the first non-asymptotic results that establish coverage and accuracy of interval length simultaneously for split conformal regression on time series.
\item For Gaussian linear processes with long memory, we establish a sharper rate for the length error of CMR (Theorem~\ref{cor:rates}): with blocks of comparable size, the calibrated length converges faster than the estimated center itself. A matching lower bound shows that this rate is exact when the calibration block is large relative to the training block (Proposition~\ref{prop:lower}). To our knowledge, Theorem~\ref{cor:rates} and Proposition~\ref{prop:lower} give the first theoretical analysis of conformal interval length dedicated to long memory and the first exact rate for the length of a conformal interval under dependence.
\end{enumerate}

Section~\ref{sec:setup} states the two procedures and introduces the functional dependence measure, Section~\ref{sec:main} states the three conditions, the coverage and length theorems and the examples, Section~\ref{sec:gaussian} contains the results on estimated centers under Gaussian long memory, Section~\ref{sec:experiments} presents the numerical studies, and Section~\ref{sec:discussion} discusses the scope of the results; all proofs are deferred to the appendices.

\section{Problem Setup and Prerequisites}\label{sec:setup}

Let $Z_t=(X_t,Y_t)=G(\ldots,\varepsilon_{t-1},\varepsilon_t)$ be a strictly stationary process driven by \iid\ innovations $\varepsilon_t$ through a measurable map $G$, and let $\Fcal_t=\sigma(\varepsilon_s:s\le t)$ be the information available at time $t$. Causal linear processes and many nonlinear time series models have this form \citep{wu2005nonlinear}. The level $\alpha\in(0,1)$ is fixed. The predictor is trained on $Z_1,\ldots,Z_n$, calibrated on $Z_{n+1},\ldots,Z_{n+m}$, and used to predict $Y_{n+m+1}$ from $X_{n+m+1}$; no observations are discarded between the blocks, and the training algorithm may use independent randomness.

\subsection{Conformalized Quantile Regression}\label{sec:cqr}

Let $\ell_\alpha(x)$ and $u_\alpha(x)$ be the $\alpha/2$- and $(1-\alpha/2)$-quantiles of the conditional law of $Y_t$ given $X_t=x$ under the stationary distribution, assumed continuous at these two quantiles, so that the oracle interval $\CQ(x)=[\ell_\alpha(x),u_\alpha(x)]$ has conditional coverage $1-\alpha$. This oracle conditions on the covariate, which may include lagged responses, and not on the whole past of the series. On the training block, any algorithm produces two functions $\lhat$ and $\uhat$ that estimate $\ell_\alpha$ and $u_\alpha$; on the calibration block, each observation receives the nonconformity score $\ShatQ_t=\max\{\lhat(X_t)-Y_t,\ Y_t-\uhat(X_t)\}$, $t=n+1,\ldots,n+m$. The calibrated correction is the $k$-th smallest score $\ShatQ_{(k)}$ with $k=\lceil(m+1)(1-\alpha)\rceil$, which is at most $m$ when $m\ge(1-\alpha)/\alpha$, as assumed throughout, and the prediction interval for $Y_{n+m+1}$ is $\ChatQ(x)=[\lhat(x)-\ShatQ_{(k)},\ \uhat(x)+\ShatQ_{(k)}]$ at $x=X_{n+m+1}$, an inverted interval being empty. The same steps with the true quantiles produce the oracle scores $\SQ_t=\max\{\ell_\alpha(X_t)-Y_t,\ Y_t-u_\alpha(X_t)\}$, whose distribution function $F^Q$ satisfies $F^Q(0)=1-\alpha$, and the $k$-th smallest oracle score $\SQ_{(k)}$ of the calibration block, an empirical $(1-\alpha)$-quantile of $m$ dependent scores that estimates zero. The training step enters the results only through the uniform error of the fitted endpoints, $\EQ=\max\{\norm{\lhat-\ell_\alpha}_\infty,\norm{\uhat-u_\alpha}_\infty\}$, the norms being taken over a fixed domain that contains every realized covariate almost surely.

For CQR we ask whether the marginal coverage $\Prob\{Y_{n+m+1}\in\ChatQ(X_{n+m+1})\}$ is close to $1-\alpha$, and whether the length of $\ChatQ(x)$ is close to the length of the oracle interval $\CQ(x)$, uniformly in $x$. On every realization,
\begin{equation}\label{eq:comparison}
\abs{\ShatQ_{(k)}-\SQ_{(k)}}\le\EQ,\qquad \sup_x\abs{\len\ChatQ(x)-\len\CQ(x)}\le4\EQ+2\abs{\SQ_{(k)}}.
\end{equation}
Inequality \eqref{eq:comparison}, proved in Appendix~\ref{app:pathwise}, uses no independence between the blocks and no gap between them, and it reduces both questions to a training question, whether $\EQ$ is small, and a calibration question, whether the empirical $(1-\alpha)$-quantile of $m$ dependent oracle scores is close to the population quantile.

\subsection{Conformalized Median Regression}\label{sec:cmr}

For CMR, the training block produces one function $\muhat$ that estimates the conditional median $\mu(x)$ of $Y_t$ given $X_t=x$, the calibration scores are the absolute residuals $\ShatM_t=\abs{Y_t-\muhat(X_t)}$, the calibrated radius is $\ShatM_{(k)}$, and the prediction interval is $\ChatM(x)=[\muhat(x)-\ShatM_{(k)},\ \muhat(x)+\ShatM_{(k)}]$. Its oracle counterpart is $\CM(x)=[\mu(x)-q_\alpha,\ \mu(x)+q_\alpha]$, where $q_\alpha$ is the $(1-\alpha)$-quantile of the oracle score $\SM_t=\abs{Y_t-\mu(X_t)}$ under the stationary law, and the training error is $\EM=\norm{\muhat-\mu}_\infty$. The two oracles differ: $\CQ(x)$ has coverage $1-\alpha$ at every $x$, whereas $\CM(x)$ uses one radius for all $x$ and has coverage $1-\alpha$ on average over $x$.

The oracle calibration error is now the distance of $\SM_{(k)}$ from $q_\alpha$ in place of the distance of $\SQ_{(k)}$ from zero, and a proof similar to that of \eqref{eq:comparison} gives, on every realization,
\begin{equation}\label{eq:comparisonM}
\abs{\ShatM_{(k)}-\SM_{(k)}}\le\EM,\quad \sup_x\abs{\len\ChatM(x)-\len\CM(x)}=2\abs{\ShatM_{(k)}-q_\alpha}\le2\EM+2\abs{\SM_{(k)}-q_\alpha}.
\end{equation}

\subsection{Functional Dependence Measure}\label{sec:fdm}

The calibration question requires a measure of the dependence of the series, and we use the functional dependence measure of \citet{wu2005nonlinear}. Let $\varepsilon_0'$ be an independent copy of $\varepsilon_0$, and let $Z_t^*=G(\ldots,\varepsilon_{-1},\varepsilon_0',\varepsilon_1,\ldots,\varepsilon_t)$ be the trajectory obtained by replacing the single innovation at time $0$ by its copy, all other innovations being retained; $\SQs_t$ and $\SMs_t$ denote the oracle scores computed from $Z_t^*=(X_t^*,Y_t^*)$ with the oracle endpoints evaluated at $X_t^*$. The functional dependence measures of the observations and of the two oracle scores at lag $j$ are
\[
\delta^Z_q(j)=\norm{Z_j-Z_j^*}_q,\qquad \delta^Q_q(j)=\norm{\SQ_j-\SQs_j}_q,\qquad \delta^M_q(j)=\norm{\SM_j-\SMs_j}_q,\qquad q\ge1,
\]
where $\norm{V}_q=(\E\abs{V}^q)^{1/q}$ with $\abs{\cdot}$ the Euclidean norm: the $L^q$ distance by which the series, or its score, $j$ steps ahead responds to one innovation. For a linear process $Y_t=\sum_{j\ge0}a_j\varepsilon_{t-j}$ the measure equals $\abs{a_j}\,\norm{\varepsilon_0-\varepsilon_0'}_q$, for a recursion $Z_t=H(Z_{t-1},\varepsilon_t)$ that contracts in $L^q$ it decays geometrically \citep{wu2004limit}, and when $a_j\asymp j^{-\beta}$ with $1/2<\beta<1$, so that the autocorrelation at lag $h$ decays like $h^{1-2\beta}$ and the series has long memory, it decays like $j^{-\beta}$: the exponent of its polynomial decay is the memory exponent of the series.

The oracle calibration error is controlled through the events $\{\SQ_t\le s\}$ for thresholds $s$ near the population quantile, and the relevant quantity is therefore the response of the conditional probability of such an event to one innovation, after the intervening innovations have been averaged out:
\begin{equation}\label{eq:omega}
\omega^Q(j)=\sup_{s\in\R}\,\big\|\Prob(\SQ_j\le s\mid\Fcal_0)-\Prob(\SQs_j\le s\mid\Fcal_0^*)\big\|_2,\qquad j\ge0,
\end{equation}
where $\Fcal_0^*=\sigma(\varepsilon_0',\varepsilon_s:s\le-1)$, and $\omega^M(j)$ is defined in the same way with $\SM_j$ and $\SMs_j$. The second probability is the first one computed from the coupled innovations, so $\omega^Q(j)$ is the supremum over $s$ of the functional dependence measure, in $L^2$, of the conditional probability $\Prob(\SQ_j\le s\mid\Fcal_0)$, which is the predictive dependence measure of the indicator of $\{\SQ_j\le s\}$ in the sense of \citet{wu2005nonlinear}.

\section{Main Results}\label{sec:main}

This section provides the marginal coverage and the interval length guarantees of the two procedures simultaneously, for the same calibrated interval. Our theoretical guarantee relies on three assumptions. The first assumption concerns the training step and is the only requirement on the learner. Throughout, the training error $E_n$ denotes $\EQ$ for CQR and $\EM$ for CMR.

\begin{assumption}[Consistent endpoints]\label{ass:train}
There are deterministic sequences $e_n\to0$ and $\xi_n\to0$ such that the training error satisfies $\Prob(E_n>e_n)\le\xi_n$ for every $n$.
\end{assumption}

Assumption~\ref{ass:train} is equivalent to $E_n\to0$ in probability, with $e_n$ and $\xi_n$ recording its rate. It is stated for a general learner on purpose and is easy to check: by ergodicity, linear quantile regression satisfies it under mild conditions on every stationary series of the form of Section~\ref{sec:setup} whose conditional quantiles are linear in bounded covariates, whatever its dependence (Appendix~\ref{app:training}), and in the Gaussian model of Section~\ref{sec:gaussian} the sample mean and the empirical quantiles of the training block satisfy it with an explicit rate.

The second assumption requires that the oracle score have probability mass on both sides of its $(1-\alpha)$-quantile at a rate bounded above and below. Let $F$ denote the distribution function of the oracle score centered at that quantile, so that $F(0)=1-\alpha$: $F=F^Q$ for CQR, and $F=F^M$ for CMR, where $F^M(s)=\Prob(\SM_t-q_\alpha\le s)$.

\begin{assumption}[Regular oracle score distribution]\label{ass:local}
There are constants $r>0$ and $0<f_\alpha\le L_\alpha<\infty$ such that $F$ is continuous on $[-r,r]$ and, for $0\le t\le r$,
\begin{enumerate}
\item[(a)] $F(t)-(1-\alpha)\ge f_\alpha t$ and $(1-\alpha)-F(-t)\ge f_\alpha t$;
\item[(b)] $F(t)-(1-\alpha)\le L_\alpha t$ and $(1-\alpha)-F(-t)\le L_\alpha t$.
\end{enumerate}
\end{assumption}

For CQR, Assumption~\ref{ass:local} holds when the conditional density of $Y_t$ given $X_t$ is bounded above and away from zero near the two oracle endpoints (Appendix~\ref{app:margin}), a condition of the kind used for efficiency under exchangeability \citep{lei2018distribution,yao2026nonasymptotic}. The length results need only part~(a); the coverage results need both parts.

The third assumption is the dependence condition. Consider the predictive dependence measure $\omega(j)$ of the thresholded oracle scores defined in \eqref{eq:omega}, which is $\omega^Q(j)$ for CQR and $\omega^M(j)$ for CMR. The following assumption controls the contribution of the oracle calibration error to the coverage and length errors.

\begin{assumption}[Functional dependence of the oracle scores]\label{ass:fdm}
There are constants $A>0$ and $\gamma>1/2$ such that $\omega(j)\le A(1\vee j)^{-\gamma}$ for all $j\ge0$.
\end{assumption}

Assumption~\ref{ass:fdm} requires no mixing, can be verified from the model through the functional dependence measure of the scores, and in principle allows for long memory (Appendix~\ref{app:smoothingproof}). It holds for the autoregression $W_t=W_{t-1}/2+B_t$ with Bernoulli innovations and for Gaussian long memory, which violate the mixing conditions of the coverage literature, and for long memory in the scale of a volatility model and contractive nonlinear recursions (Appendices~\ref{app:examples} and~\ref{app:switch}). For the Bernoulli autoregression, the coverage bounds of \citet{oliveira2024split} and \citet{barber2026predictive} are vacuous, because its $\beta$-mixing coefficients, on which the first bound rests, and its switch coefficients, on which the second rests, equal one at every lag (Appendix~\ref{app:switch}).

Under Assumption~\ref{ass:fdm}, the empirical distribution function of the oracle scores concentrates near the threshold at the rate $\rho_m=m^{-1/2}$ if $\gamma>1$, $\rho_m=m^{-1/2}\log(m+1)$ if $\gamma=1$ and $\rho_m=m^{1/2-\gamma}$ if $1/2<\gamma<1$, with a constant $C_\sigma$ depending only on $A$ and $\gamma$ (Lemma~\ref{lem:sigmarate} in Appendix~\ref{app:projection}); apart from the training error, this is the only way in which the dependence of the series enters the results. For brevity, we first state the marginal coverage result and the interval length guarantee for CQR.

\begin{theorem}[CQR marginal coverage]\label{thm:coverage}
Suppose that Assumptions~\ref{ass:train}--\ref{ass:fdm} hold. There exist $m_0\le m_1$, depending only on $\alpha$, $\gamma$, $f_\alpha$, $L_\alpha$, $r$ and $C_\sigma$, such that for all $n$ with $e_n\le r/4$ and all $m\ge m_1$,
\begin{equation}\label{eq:coveragerate}
\big|\Prob\{Y_{n+m+1}\in\ChatQ(X_{n+m+1})\}-(1-\alpha)\big|\le 2L_\alpha e_n+\xi_n+C_1\rho_m^{2/3},
\end{equation}
where $C_1$ depends only on $L_\alpha$, $f_\alpha$ and $C_\sigma$. If in addition the conditional law of $\SQ_{n+m+1}$ given $\Fcal_{n+m}$ has a density at most $L_{\mathrm{pred}}$ almost surely and $\E(\SQ_0)^2<\infty$, then for all $n$ and all $m\ge m_0$,
\begin{equation}\label{eq:coveragepred}
\big|\Prob\{Y_{n+m+1}\in\ChatQ(X_{n+m+1})\}-(1-\alpha)\big|\le L_{\mathrm{pred}}\{2\E\EQ+C_2\rho_m/2\},
\end{equation}
where $C_2$ depends only on $\alpha$, $f_\alpha$, $r$, $\E(\SQ_0)^2$ and $C_\sigma$.
\end{theorem}

Theorem~\ref{thm:coverage} requires no mixing condition and, in principle, allows for long memory. Bound \eqref{eq:coveragerate} holds without a density for the score at the test time; under the additional density condition, bound \eqref{eq:coveragepred} replaces $\rho_m^{2/3}$ by $\rho_m$ (Appendix~\ref{app:jointproof}).

\begin{theorem}[CQR accuracy of interval length]\label{thm:length}
Suppose that Assumptions~\ref{ass:local}(a) and \ref{ass:fdm} hold. For every $e\ge0$ and every $t$ with $4/(mf_\alpha)\le t\le r$, a range that is nonempty when $m\ge4/(f_\alpha r)$,
\begin{equation}\label{eq:lengthtail}
\Prob\Big\{\sup_x\abs{\len\ChatQ(x)-\len\CQ(x)}>4e+2t\Big\}\le\Prob(\EQ>e)+\frac{8C_\sigma^2\rho_m^2}{f_\alpha^2t^2}.
\end{equation}
If $\E(\SQ_0)^2<\infty$, then for all $m\ge m_0$,
\begin{equation}\label{eq:lengthexpectation}
\E\sup_x\abs{\len\ChatQ(x)-\len\CQ(x)}\le4\E\EQ+C_2\rho_m,
\end{equation}
with the constant $C_2$ of Theorem~\ref{thm:coverage}.
\end{theorem}

The proofs of Theorems~\ref{thm:coverage} and \ref{thm:length} are given in Appendix~\ref{app:jointproof}. To our knowledge, Theorem~\ref{thm:length} is the first non-asymptotic comparison of the calibrated length with the oracle length under temporal dependence. Together, the two theorems show that under Assumptions~\ref{ass:train}--\ref{ass:fdm}, as $n\to\infty$ and $m\to\infty$ in any manner, the same calibrated interval is both asymptotically valid and asymptotically as short as the oracle interval.

The CMR versions are as follows.

\begin{theorem}[CMR marginal coverage and accuracy of interval length]\label{thm:cmr}
Suppose that Assumptions~\ref{ass:train}--\ref{ass:fdm} hold for CMR, that is, for $\EM$, $F^M$ and $\omega^M$, and let $m_0$, $m_1$, $C_1$ and $C_2$ be as in Theorems~\ref{thm:coverage}--\ref{thm:length}, with $\E(\SM_0-q_\alpha)^2$ in place of $\E(\SQ_0)^2$. For all $n$ with $e_n\le r/4$ and all $m\ge m_1$,
\begin{equation}\label{eq:cmrcoverage}
\big|\Prob\{Y_{n+m+1}\in\ChatM(X_{n+m+1})\}-(1-\alpha)\big|\le 2L_\alpha e_n+\xi_n+C_1\rho_m^{2/3},
\end{equation}
and if the conditional law of $\SM_{n+m+1}$ given $\Fcal_{n+m}$ has a density at most $L_{\mathrm{pred}}$ almost surely and $\E(\SM_0)^2<\infty$, then for all $n$ and all $m\ge m_0$,
\begin{equation}\label{eq:cmrcoveragepred}
\big|\Prob\{Y_{n+m+1}\in\ChatM(X_{n+m+1})\}-(1-\alpha)\big|\le L_{\mathrm{pred}}\{2\E\EM+C_2\rho_m/2\}.
\end{equation}
For every $e\ge0$ and every $t$ with $4/(mf_\alpha)\le t\le r$,
\begin{equation}\label{eq:cmrlengthtail}
\Prob\Big\{\sup_x\abs{\len\ChatM(x)-\len\CM(x)}>2e+2t\Big\}\le\Prob(\EM>e)+\frac{8C_\sigma^2\rho_m^2}{f_\alpha^2t^2},
\end{equation}
and if $\E(\SM_0)^2<\infty$, then for all $m\ge m_0$,
\begin{equation}\label{eq:cmrlength}
\E\sup_x\abs{\len\ChatM(x)-\len\CM(x)}=2\,\E\abs{\ShatM_{(k)}-q_\alpha}\le2\E\EM+C_2\rho_m.
\end{equation}
\end{theorem}

Theorem~\ref{thm:cmr}, proved in Appendix~\ref{app:jointproof}, gives the same guarantees for CMR, so the CMR interval is also asymptotically valid and asymptotically as short as its oracle; for Gaussian long memory, Section~\ref{sec:gaussian} sharpens its length bound.

\section{Estimated Centers under Gaussian Long Memory}\label{sec:gaussian}

A common difficulty with long-memory processes is that estimation converges slowly: the sample median, for instance, converges much more slowly than under short memory \citep{dehling1989empirical}. The length bound of Theorem~\ref{thm:cmr} inherits this slowness through both of its terms, the training error and the calibration error. In this section, devoted to long-memory processes, we show that the calibrated length can nevertheless converge faster than the estimated center. We study the Gaussian linear process. It is more convenient to analyze and is nevertheless representative: by the Wold decomposition, every stationary Gaussian process without a deterministic component can be written as a Gaussian linear process \citep[Section~5.7]{brockwell1991time}. Specifically, we consider
\begin{equation}\label{eq:linearprocess}
G_t=\sum_{j\ge0}a_j\varepsilon_{t-j},\quad \varepsilon_t\overset{\iid}{\sim}N(0,1),\quad \sum_{j\ge0}a_j^2=1,\quad a_0\ne0,\quad \abs{a_j}\le A(1+j)^{-\beta},
\end{equation}
with $1/2<\beta<1$, the long-memory range of Section~\ref{sec:fdm}. Gaussian innovations are assumed for convenience; since the improvement rests on the symmetry of their law, we expect it to extend to symmetric non-Gaussian innovations.

We observe $Y_t=\mu+G_t$ with an unknown location $\mu$ and a constant covariate. The oracle intervals of CQR and CMR then coincide, since the law of $G_t$ is symmetric: $\CQ=\CM=[\mu-z_\alpha,\ \mu+z_\alpha]$, where $z_\alpha=\Phi^{-1}(1-\alpha/2)$ is the radius $q_\alpha$ of Section~\ref{sec:cmr}. The two procedures also coincide: for any fitted endpoints, the CQR score is $\max\{\lhat-y,\ y-\uhat\}=\abs{y-\widehat c}-(\uhat-\lhat)/2$ with the midpoint $\widehat c=(\lhat+\uhat)/2$; because the covariate is constant, the half-width $(\uhat-\lhat)/2$ is common to all scores and cancels, and the CQR interval is the CMR interval $[\widehat c-\ShatM_{(k)},\ \widehat c+\ShatM_{(k)}]$ with center $\widehat c$. In this section we therefore provide the theory for CMR, with a center fitted on the training block, and it applies to CQR with the midpoint of the fitted endpoints as the center.

For the sample mean, the two terms of Theorem~\ref{thm:cmr} are of order $n^{1/2-\beta}$ and $m^{1/2-\beta}$ (Appendix~\ref{app:examples}), and in this model both can be improved. Because the centered oracle score $\abs{G_t}-z_\alpha$ is an even function of $G_t$, the calibration rate improves from $m^{1/2-\beta}$ to the following rate (Appendix~\ref{app:hermite}):
\begin{equation}\label{eq:rm}
r_m(\beta)=\begin{cases}m^{-(2\beta-1)}, & 1/2<\beta<3/4,\\ \{\log(m+1)/m\}^{1/2}, & \beta=3/4,\\ m^{-1/2}, & 3/4<\beta<1.\end{cases}
\end{equation}
Because the law of $G_t$ is symmetric, an error in the center changes the population radius of the calibration residuals only at second order (Appendix~\ref{app:shifted}). The difficulty is to retain this second-order effect with a finite calibration block that is adjacent to the training block, so that the center is correlated with the calibration residuals.

\begin{theorem}[CMR length under Gaussian long memory]\label{cor:rates}
Assume \eqref{eq:linearprocess}, and let $\muhat$ be a measurable function of $Y_1,\ldots,Y_n$ and of randomness independent of the series. There are constants $m_0$, depending only on $\alpha$, and $C$, depending only on $\alpha$, $\beta$, $A$ and $a_0$, such that for all $n$ and all $m\ge m_0$,
\begin{equation}\label{eq:secondorder}
\E\abs{\len\ChatM-2z_\alpha}\le C\big\{\E(\muhat-\mu)^2+r_m(\beta)\big\}.
\end{equation}
For the sample mean, the empirical median, and the midpoint of the empirical $\alpha/2$- and $(1-\alpha/2)$-quantiles of the training block, $\E(\muhat-\mu)^2\le Cn^{-(2\beta-1)}$, and therefore $\E\abs{\len\ChatM-2z_\alpha}\le C\{n^{-(2\beta-1)}+r_m(\beta)\}$.
\end{theorem}

Theorem~\ref{cor:rates} is proved in Appendix~\ref{app:ratesproof}. Under this model, the center enters the length through its mean squared error, instead of its absolute error as in Section~\ref{sec:main}, and with the midpoint as center the theorem covers CQR with the empirical training quantiles as fitted endpoints. With blocks of comparable size, $n\asymp m$, the length error is at most of the order $r_m(\beta)$, which Theorem~\ref{cor:rates} gives for a known center, although the center converges only at the rate $n^{1/2-\beta}$ when $a_j\asymp j^{-\beta}$; for $\beta>3/4$, the rate $m^{-1/2}$ of independent scores survives both the long memory and the estimated center. To our knowledge, this is the first theoretical analysis of conformal interval length dedicated to long memory. The coverage error of the interval is at most of the same order as the right-hand side of \eqref{eq:secondorder} (Appendix~\ref{app:coverageproof}), so the interval also attains the nominal level at this rate.

\begin{proposition}[Lower bound]\label{prop:lower}
Assume \eqref{eq:linearprocess}, and let $\muhat$ be as in Theorem~\ref{cor:rates}. For all $n$ and all $m\ge m_0$, with $m_0$ as in Theorem~\ref{cor:rates},
\begin{equation}\label{eq:lowerrisk}
\E\abs{\len\ChatM-2z_\alpha}\ge c_1\E\min\{(\muhat-\mu)^2,1\}-C\big[r_m(\beta)+\{m^{-(2\beta-1)}\E(\muhat-\mu)^2\}^{1/2}\big],
\end{equation}
where $c_1>0$ depends only on $\alpha$ and $C$ only on $\alpha$, $\beta$, $A$ and $a_0$. If, in addition, $a_j\ge c(1+j)^{-\beta}$ for all $j\ge j_0$, with $c>0$, and $\muhat=T(Y_1,\ldots,Y_n)$, where $T$ is nondecreasing in each coordinate, satisfies $T(y_1+h,\ldots,y_n+h)=T(y_1,\ldots,y_n)+h$ for every real $h$, and has $\E(\muhat-\mu)^2\le C_gn^{-(2\beta-1)}$, then for all $n\ge n_0$ and $m\ge Kn^{\max(1,4\beta-2)}\log(n+1)$,
\begin{equation}\label{eq:lowerbound}
\E\abs{\len\ChatM-2z_\alpha}\ge c_3\big\{n^{-(2\beta-1)}+r_m(\beta)\big\},
\end{equation}
where $n_0$, $K$ and $c_3>0$ depend only on $\alpha$, $\beta$, $A$, $a_0$, $c$, $j_0$ and $C_g$.
\end{proposition}

Proposition~\ref{prop:lower} is proved in Appendix~\ref{app:lowerproof}. The sample mean, the empirical median and the midpoint satisfy its conditions on the center, so Theorem~\ref{cor:rates} and Proposition~\ref{prop:lower} together show that $n^{-(2\beta-1)}+r_m(\beta)$ is the exact rate of their length error when the calibration block is large relative to the training block. By Proposition~\ref{prop:lower}, no monotone and translation-equivariant center whose mean squared error is of the same order $n^{-(2\beta-1)}$ does better. To our knowledge, Theorem~\ref{cor:rates} and Proposition~\ref{prop:lower} give the first exact rate for the length of a conformal interval under dependence.

\section{Numerical Studies}\label{sec:experiments}

We examine the theory in three simulation studies whose oracle intervals are known: the coverage and the length in four dependence setups, two of which the mixing theory does not cover (Study~1), the sharper length rate under long memory (Study~2) and the lower bound of Proposition~\ref{prop:lower} (Study~3). Throughout, $\alpha=.1$, the training and calibration blocks are adjacent, each configuration uses $1000$ independent trajectories, which give Monte Carlo standard errors and pointwise 95\% intervals, and each procedure is compared with its own oracle interval. Appendix~\ref{app:experiments} gives the designs and the complete results.

\begin{figure}[t]
\begin{center}
\includegraphics[width=.9\linewidth]{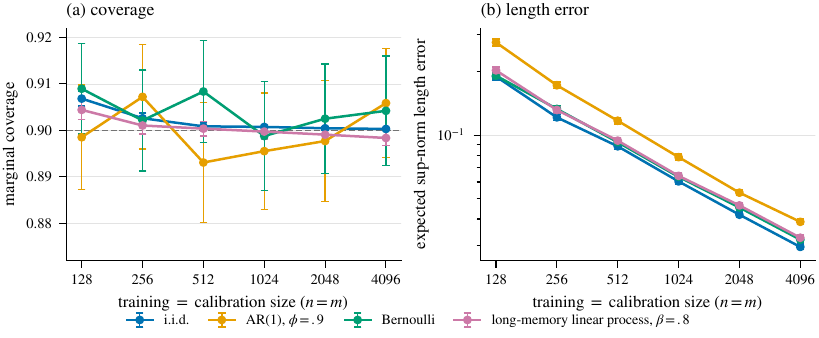}
\end{center}
\caption{Study~1, CQR with quadratic quantile regression, $n=m\in\{128,\ldots,4096\}$. (a) Marginal coverage; the dashed line is the nominal level. (b) Expected sup-norm length error, on a logarithmic scale. Pointwise 95\% Monte Carlo intervals.}
\label{fig:g1}
\end{figure}

\paragraph{Study 1.} This study illustrates that split conformal regression works well on processes beyond the existing theory. The observations are $Y_t=f(X_t)+s(X_t)h(U_t)$ with $f(x)=x+x^2$, $s(x)=1+.3x^2$ and $h(u)=u+u^2/2$, where the covariates $X_t$ are \iid\ uniform on $[-1,1]$ and independent of $(U_t)$. The temporal dependence is controlled by $(U_t)$, and we consider four setups for it: independent variables, an AR(1) process, the Bernoulli autoregression of Section~\ref{sec:main} and a long-memory linear process of the form \eqref{eq:linearprocess} with $\beta=.8$, each transformed to the uniform law on $[0,1]$. The exchangeable theory covers only the independent setup and the mixing theory also covers AR(1), but neither covers the Bernoulli setup or the long-memory linear process. Assumptions~\ref{ass:train}--\ref{ass:fdm} hold in all four setups, so the theory of this paper covers all of them. The procedure is CQR with quadratic quantile regression at the levels $.05$ and $.95$, which is correctly specified. For $n=m\in\{128,\ldots,4096\}$, we estimate the marginal coverage and the expected sup-norm length error $\E\sup_x\abs{\len\ChatQ(x)-\len\CQ(x)}$ of Theorem~\ref{thm:length}. Appendix~\ref{app:study1} gives the details and the paired CMR results.

Figure~\ref{fig:g1} shows that the marginal coverage stays within $.01$ of the nominal level in all four setups, including the two that only the theory of this paper covers, and that the length error decreases toward zero in all four. For the long-memory linear process, the calibration rate $\rho_m$ of Theorem~\ref{thm:length} is $m^{-.3}$, against $m^{-1/2}$ in the other three setups, but over the tested sizes its length error decreases at about the same rate as theirs, with fitted decay exponents between $.51$ and $.53$ in all four setups, whereas for this process the length error of the paired CMR interval decreases like $m^{-.3}$ (Appendix~\ref{app:study1}). The same contrast appears in the oracle calibration step: in this setup, $U_t=\Phi(G_t)$ with $G_t$ of the form \eqref{eq:linearprocess}, and the CQR oracle interval contains $Y_t$ exactly when $\abs{G_t}\le z_\alpha$, an event that is symmetric in $G_t$, whereas the corresponding event for the CMR oracle interval is not, because $h$ is not linear. This agrees with the motivation of Section~\ref{sec:gaussian}: long-memory linear processes require a separate analysis to obtain a sharper rate.

\begin{figure}[t]
\begin{center}
\includegraphics[width=.9\linewidth]{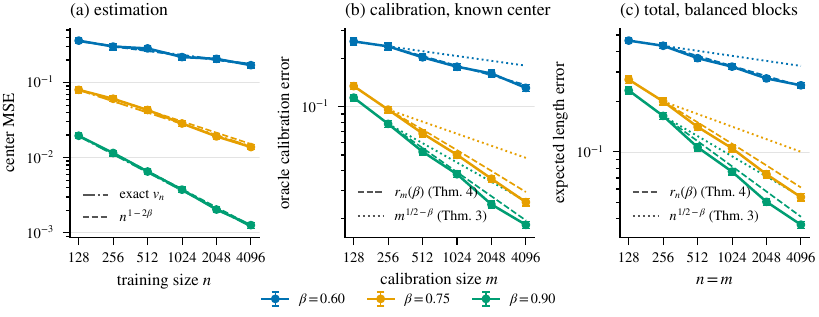}
\end{center}
\caption{Study~2, CMR with the sample-mean center, Gaussian linear process \eqref{eq:linearprocess} with $\beta\in\{.6,.75,.9\}$. (a) Mean squared error of the center, with its exact value $v_n$ (dash-dotted) and $n^{-(2\beta-1)}$ (dashed). (b) Oracle calibration error $\E\abs{\qRor-z_\alpha}$, with $r_m(\beta)$ (dashed) and $m^{1/2-\beta}$ (dotted). (c) Length error $\E\abs{\len\ChatM-2z_\alpha}$ with $n=m$, with $r_n(\beta)$ (dashed) and $n^{1/2-\beta}$ (dotted). The guides, matched to the curves at $256$, compare decay, not constants. Pointwise 95\% Monte Carlo intervals.}
\label{fig:l1}
\end{figure}

\paragraph{Study 2.} This study examines the model of Section~\ref{sec:gaussian} for two purposes: to measure how slowly the center converges under long memory, and to check whether the length error nevertheless attains the rate of Theorem~\ref{cor:rates}. The observations are $Y_t=\mu+G_t$, where $G_t$ is a Gaussian linear process of the form \eqref{eq:linearprocess} whose coefficients $a_j$ are asymptotically proportional to $j^{-\beta}$, with $\beta\in\{.6,.75,.9\}$, one value in each case of \eqref{eq:rm}. Appendix~\ref{app:conventions} gives the coefficients and the exact simulation of the paths. The procedure is CMR with the training sample mean as the center. For $n=m\in\{128,\ldots,4096\}$, we measure three quantities on the same trajectories: the mean squared error of the center, the oracle calibration error $\E\abs{\qRor-z_\alpha}$ and the interval length error $\E\abs{\len\ChatM-2z_\alpha}$. Figure~\ref{fig:l1} compares them with the rates of Theorems~\ref{thm:cmr} and~\ref{cor:rates}.

There are two findings. First, the center converges slowly: its mean squared error decreases like $n^{-(2\beta-1)}$, much more slowly than under short memory, in agreement with its exact value. Second, the oracle calibration error and the length error follow the rates $r_m(\beta)$ and $r_n(\beta)$ of Theorem~\ref{cor:rates}, and not the slower rates $m^{1/2-\beta}$ and $n^{1/2-\beta}$ of Theorem~\ref{thm:cmr}. We repeated the study with two non-Gaussian innovation laws, keeping the coefficients and taking the training median as the center. With Laplace innovations, which are symmetric, the findings are the same, as Section~\ref{sec:gaussian} expects. Centered exponential innovations are asymmetric by design, to test whether the improvement of Section~\ref{sec:gaussian} rests on symmetry: at $\beta=.75$ and $.9$, the length error then follows the rate of Theorem~\ref{thm:cmr} more closely than that of Theorem~\ref{cor:rates}. At $\beta=.6$ it still follows Theorem~\ref{cor:rates} over the tested sizes, because the stronger memory spreads each innovation over many lags and brings the law of the series closer to a Gaussian law (Appendix~\ref{app:extension}).

\begin{figure}[t]
\begin{center}
\includegraphics[width=.9\linewidth]{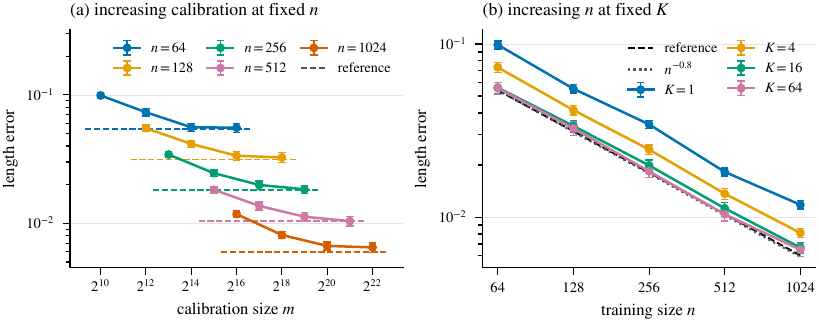}
\end{center}
\caption{Study~3, CMR with the sample-mean center, $\beta=.9$ and $m\approx Kn^{1.6}$. (a) Length error against $m$ for each training size $n$, with the reference of each $n$ (dashed, same color). (b) The same errors against $n$ for each $K$, with the reference (dashed) and the rate $n^{-.8}$ matched to the reference at $n=64$ (dotted). Pointwise 95\% Monte Carlo intervals.}
\label{fig:s3}
\end{figure}

\paragraph{Study 3.} This study examines the lower bound of Proposition~\ref{prop:lower}. Calibration does not remove the contribution from center estimation, so however long the calibration block, the length error remains at least of the order $n^{-(2\beta-1)}$, which is the order of the mean squared error of the center. We keep the model of Study~2 with $\beta=.9$ and the sample-mean center, and for each training size $n\in\{64,\ldots,1024\}$ we take the calibration sizes $m\approx Kn^{1.6}$ with $K\in\{1,4,16,64\}$, which follow the threshold of Proposition~\ref{prop:lower} up to its logarithmic factor and reach thousands of times the training size. The reference is the length error of the interval calibrated on an infinitely long block, which depends only on the center and is close to $z_\alpha$ times its mean squared error. Figure~\ref{fig:s3} shows that once the calibration block is long enough, the length error stops decreasing and stays within 8\% of this reference at every training size: more calibration data only reduces the error due to calibration, which halves with each fourfold increase of $m$, while the reference, set by the center, does not move. The reference decreases at the rate $n^{-.8}=n^{-(2\beta-1)}$: from $n=64$ to $1024$, the length error with the largest calibration blocks falls by a factor of $8.6$, against $9.2$ for $n^{-.8}$. Studies~2 and~3 thus illustrate both halves of the exact rate of Section~\ref{sec:gaussian}: the length error attains the rate of Theorem~\ref{cor:rates}, and the contribution of the center, which enters at second order, remains however much calibration data is added. Appendix~\ref{app:study3} gives the design, the computation of the reference and the complete results.

\paragraph{Real Data Analysis.} We also apply CQR and CMR to day-ahead forecast errors of electricity load, which have long memory, and check both the coverage and the interval length. Both procedures cover close to the nominal level, and the mean length of CMR is within 2\% of a proxy for its oracle length computed in hindsight (Appendix~\ref{app:application}).

\section{Discussion}\label{sec:discussion}

The functional dependence measure replaces mixing as the theoretical toolbox for split conformal regression on time series: it is computed from the model that generates the series, the guarantees hold for processes that are not mixing, and the same calibrated interval satisfies non-asymptotic coverage and length guarantees. Under Gaussian long memory, the length converges faster than the estimated center, at a rate that a matching lower bound shows to be exact for long calibration blocks.

The theory has two limitations. First, the coverage guarantees are marginal and concern one split of a stationary series; they extend to any fixed forecast horizon, but rolling refits and nonstationary series are not covered. Second, the sharper rate of Section~\ref{sec:gaussian} is proved for Gaussian linear processes with a constant covariate and rests on the symmetry of their law; Study~2 supports symmetric non-Gaussian innovations numerically, and a proof for them would replace the Hermite expansions of Appendix~\ref{app:hermite} by the expansions of the empirical process of linear processes \citep{wu2003empirical}.

\clearpage
\bibliography{references}
\bibliographystyle{iclr2027_conference}

\clearpage
\appendix
\section{Proofs for Sections~\ref{sec:setup} and \ref{sec:main}}\label{app:quantile}

This appendix proves the pathwise comparison \eqref{eq:comparison}--\eqref{eq:comparisonM}, the bound on the calibration variance stated before Theorem~\ref{thm:coverage}, and Theorems~\ref{thm:coverage}--\ref{thm:cmr} together with the convergence statements after Theorem~\ref{thm:length}. Appendix~\ref{app:assumptions} verifies the three assumptions for the learners and models of the paper.

The analysis of the calibration step combines known tools. The variance bound for the empirical distribution function of the oracle scores (Lemma~\ref{lem:sigmarate}) uses the method of \citet{wu2005nonlinear}, who bounds partial sums of a causal process, its empirical distribution function included, through their projections onto individual innovations and the predictive dependence measure; when this measure is summable, the method gives $\rho_m=m^{-1/2}$, and Proposition~\ref{prop:convolution} extends the bound to polynomial decay, which gives the two slower regimes of $\rho_m$. The pathwise comparison \eqref{eq:comparison}--\eqref{eq:comparisonM} is the deterministic comparison also used by \citet{yao2026nonasymptotic} for independent observations, and the two-threshold bound of Appendix~\ref{app:twothreshold} is the Chebyshev bound for an empirical quantile. What is new is the formulation through the predictive dependence measure of the thresholded scores, with the transfer conditions of Proposition~\ref{prop:sufficient}; the non-asymptotic combination of the two steps, including the window argument of Appendix~\ref{app:coverage}, which gives coverage without a density for the score at the test time; and the length statements.

\subsection{Proof of Equations \texorpdfstring{\eqref{eq:comparison} and \eqref{eq:comparisonM}}{(1) and (2)}}\label{app:pathwise}

\begin{proof}[Proof of \eqref{eq:comparison}]
For any four real numbers, $\abs{\max(a,b)-\max(c,d)}\le\max(\abs{a-c},\abs{b-d})$. Applying this inequality to the two endpoint residuals gives $\abs{\ShatQ_t-\SQ_t}\le\EQ$ at every calibration time and at the test time. The $k$-th smallest values then differ by at most $\EQ$, also when ties occur: at least $k$ of the oracle scores $\SQ_{n+1},\ldots,\SQ_{n+m}$ are at most $\SQ_{(k)}$, the corresponding fitted scores are at most $\SQ_{(k)}+\EQ$, and therefore $\ShatQ_{(k)}\le\SQ_{(k)}+\EQ$; interchanging the fitted and the oracle scores gives the reverse inequality, so $\abs{\ShatQ_{(k)}-\SQ_{(k)}}\le\EQ$.

For each $x$, write $w(x)=u_\alpha(x)-\ell_\alpha(x)\ge0$ and $\widehat{w}(x)=\uhat(x)-\lhat(x)$. Then $\len\CQ(x)=w(x)$ and $\len\ChatQ(x)=[\widehat{w}(x)+2\ShatQ_{(k)}]_+$. Since the positive-part map is $1$-Lipschitz,
\begin{align*}
\abs{[\widehat{w}(x)+2\ShatQ_{(k)}]_+-w(x)}&\le\abs{\widehat{w}(x)-w(x)}+2\abs{\ShatQ_{(k)}}\\
&\le 2\EQ+2\abs{\ShatQ_{(k)}-\SQ_{(k)}}+2\abs{\SQ_{(k)}}\le 4\EQ+2\abs{\SQ_{(k)}}.
\end{align*}
Taking the supremum over $x$ proves the length bound in \eqref{eq:comparison}, including the cases of quantile crossing and empty fitted intervals. Since the bound holds simultaneously for all $x$ in the stated domain, it applies to an adjacent, dependent test covariate without any train--test decoupling argument.
\end{proof}

\begin{proof}[Proof of \eqref{eq:comparisonM}]
The proof follows that of \eqref{eq:comparison}: $\abs{\ShatM_t-\SM_t}\le\abs{\muhat(X_t)-\mu(X_t)}\le\EM$ at every time, so the order-statistic step gives $\abs{\ShatM_{(k)}-\SM_{(k)}}\le\EM$; moreover $\len\ChatM(x)=2\ShatM_{(k)}$ and $\len\CM(x)=2q_\alpha$ for every $x$, and the triangle inequality completes the proof.
\end{proof}

\subsection{The Calibration Variance under Functional Dependence}\label{app:projection}

This subsection fixes the notation of the remaining proofs and bounds the variance of the empirical distribution function of the oracle scores near the threshold under Assumption~\ref{ass:fdm}.

In the remaining appendices, $\Sor_t$ denotes the centered oracle score of the procedure under study, $\SQ_t$ for CQR and $\SM_t-q_\alpha$ for CMR; $\Shat_t$ denotes the corresponding fitted score, $\ShatQ_t$ or $\ShatM_t-q_\alpha$; $\Qor$ and $\Shat_{(k)}$ are their $k$-th smallest values over the calibration block; $F$ is the distribution function of $\Sor_t$, so that $F(0)=1-\alpha$ in both cases; and $E_n$ is the training error, $\EQ$ or $\EM$. For CMR the centering cancels in the algorithm, whose radius is $\ShatM_{(k)}=q_\alpha+\Shat_{(k)}$, and $\Qor=\SM_{(k)}-q_\alpha$ is the calibration error of the oracle radius; the predictive dependence measure $\omega(j)$ of \eqref{eq:omega} is unchanged by the centering. With this notation, the proofs of \eqref{eq:comparison} and \eqref{eq:comparisonM} give, for both procedures,
\begin{equation}\label{eq:centeredcomparison}
\abs{\Shat_t-\Sor_t}\le E_n\quad(n<t\le n+m+1),\qquad \abs{\Shat_{(k)}-\Qor}\le E_n,
\end{equation}
and the length error is at most $4E_n+2\abs{\Qor}$ for CQR and equal to $2\abs{\Shat_{(k)}}\le2E_n+2\abs{\Qor}$ for CMR. The coverage argument of Appendix~\ref{app:coverage} involves $2E_n+\abs{\Qor}$ for both procedures, because both the fitted test score and the fitted calibration threshold are perturbed.

Let
\[
\Fhat(s)=\frac1m\sum_{i=1}^m\ind{\Sor_{n+i}\le s},\qquad \sigma_m^2=\sup_{\abs{s}\le r}\Var\Fhat(s),
\]
be the empirical distribution function of the oracle scores on the calibration block and its largest variance near the threshold; stationarity gives $\E\Fhat(s)=F(s)$. The bound to be proved is $\sigma_m\le C_\sigma\rho_m$ (Lemma~\ref{lem:sigmarate}). Let $\calP_jW=\E(W\mid\Fcal_j)-\E(W\mid\Fcal_{j-1})$ denote the projection onto the innovation at time $j$, and let
\[
\pi_j(s)=\norm{\calP_0\ind{\Sor_j\le s}}_2,\qquad j\ge0,
\]
be the projection norm of the thresholded score, which measures how much the innovation $\varepsilon_0$ predicts about the event $\{\Sor_j\le s\}$. Let $J_j(s)=\ind{\Sor_j\le s}$ and $J_j^*(s)=\ind{\Sor_j^*\le s}$, the latter computed from the coupled trajectory of Section~\ref{sec:fdm}, and let $W_j(s)=\Prob(\Sor_j\le s\mid\Fcal_0)$ and $W_j^*(s)=\Prob(\Sor_j^*\le s\mid\Fcal_0^*)$ be the two conditional probabilities compared in \eqref{eq:omega}, so that $\omega(j)=\sup_s\norm{W_j(s)-W_j^*(s)}_2$. Independence of the replaced innovation, of its copy and of the future innovations gives $W_j(s)=\E\{J_j(s)\mid\Fcal_0,\varepsilon_0'\}$, $W_j^*(s)=\E\{J_j^*(s)\mid\Fcal_0,\varepsilon_0'\}$ and $\E\{W_j^*(s)\mid\Fcal_0\}=\E\{W_j(s)\mid\Fcal_{-1}\}$, and therefore
\begin{equation}\label{eq:couplingidentity}
\calP_0J_j=\E(W_j-W_j^*\mid\Fcal_0)=\E(J_j-J_j^*\mid\Fcal_0),\qquad \pi_j(s)=2^{-1/2}\norm{W_j(s)-W_j^*(s)}_2;
\end{equation}
the second identity holds because $W_j(s)$ and $W_j^*(s)$ are the same function of $\Fcal_{-1}$ and of $\varepsilon_0$ or $\varepsilon_0'$, respectively, and are therefore independent and identically distributed given $\Fcal_{-1}$. Assumption~\ref{ass:fdm} thus gives $\sup_s\pi_j(s)\le2^{-1/2}A(1\vee j)^{-\gamma}$.

\begin{proposition}[Projection envelope]\label{prop:envelope}
For every $s$,
\begin{equation}\label{eq:envelope}
\Var\Big(\sum_{i=1}^m\ind{\Sor_{n+i}\le s}\Big)\le\sum_{j=-\infty}^{m}\Big(\sum_{i=\max(1,j)}^m\pi_{i-j}(s)\Big)^2\le m\Big(\sum_{j\ge0}\pi_j(s)\Big)^2.
\end{equation}
If the marginal density of $\Sor_t$ is at most $L$, then
\begin{equation}\label{eq:projectionbound}
\norm{W_j(s)-W_j^*(s)}_2\le\min\Big\{\{2\min(F(s),1-F(s))\}^{1/2},\ \sqrt{3}\,L^{q/[2(q+1)]}\delta_q(j)^{q/[2(q+1)]}\Big\},
\end{equation}
where $\delta_q(j)=\norm{\Sor_j-\Sor_j^*}_q$ is the functional dependence measure of the score.
\end{proposition}

\begin{proof}
Fix $s$ and set $I_t=\ind{\Sor_t\le s}-F(s)$. As a square-integrable causal function of \iid\ innovations, $I_t$ has the orthogonal decomposition $I_t=\sum_{j=-\infty}^t\calP_jI_t$ in $L^2$; the backward limit of the conditional expectations is zero by triviality of the tail sigma-field of the \iid\ past. Summing over $t=1,\ldots,m$ and using orthogonality across $j$ gives
\[
\Big\|\sum_{t=1}^mI_t\Big\|_2^2=\sum_{j=-\infty}^m\Big\|\sum_{t=\max(1,j)}^m\calP_jI_t\Big\|_2^2
\le\sum_{j=-\infty}^m\Big(\sum_{t=\max(1,j)}^m\pi_{t-j}(s)\Big)^2,
\]
which is the first inequality in \eqref{eq:envelope} after a shift of the time index by stationarity. Regrouping the projections by lag and applying Minkowski's inequality,
\[
\Big\|\sum_{t=1}^mI_t\Big\|_2\le\sum_{k\ge0}\Big\|\sum_{t=1}^m\calP_{t-k}I_t\Big\|_2=\sqrt m\sum_{k\ge0}\pi_k(s),
\]
where the equality uses orthogonality over the distinct projection times $t-k$; this is the second inequality.

For \eqref{eq:projectionbound}, the representation $W_j(s)-W_j^*(s)=\E\{J_j(s)-J_j^*(s)\mid\Fcal_0,\varepsilon_0'\}$ and Jensen's inequality give $\norm{W_j(s)-W_j^*(s)}_2\le\norm{J_j-J_j^*}_2$. The two indicators have identical marginal laws, so their disagreement probability is at most both $2F(s)$ and $2(1-F(s))$. For every bandwidth $h>0$, disagreement also implies either $\abs{\Sor_j-s}\le h$ or $\abs{\Sor_j-\Sor_j^*}>h$, hence
\[
\norm{J_j-J_j^*}_2^2\le2Lh+\delta_q(j)^q/h^q.
\]
For $\delta_q(j)>0$, the choice $h=\{\delta_q(j)^q/L\}^{1/(q+1)}$ makes the right-hand side $3L^{q/(q+1)}\delta_q(j)^{q/(q+1)}$; if the dependence coefficient is zero, let $h\downarrow0$. Taking square roots and combining the two bounds proves \eqref{eq:projectionbound}.
\end{proof}

\begin{proposition}\label{prop:convolution}
If $\sup_{s}\pi_j(s)\le B(1\vee j)^{-\gamma}$ for all $j\ge0$ and some $\gamma>1/2$, then $\sigma_m\le C_\gamma B\rho_m$ with
\begin{equation}\label{eq:sigmarate}
\rho_m=\begin{cases}m^{-1/2}, & \gamma>1,\\ m^{-1/2}\log(m+1), & \gamma=1,\\ m^{1/2-\gamma}, & 1/2<\gamma<1,\end{cases}
\end{equation}
and $C_\gamma$ depending only on $\gamma$.
\end{proposition}

\begin{proof}
Since $(1\vee j)^{-\gamma}\le2^\gamma(j+1)^{-\gamma}$, the envelope is at most $2^\gamma B(j+1)^{-\gamma}$. For projection indices $-m<j\le m$, each inner sum in \eqref{eq:envelope} is at most $2^\gamma B\sum_{k=0}^{2m}(k+1)^{-\gamma}$, and there are at most $2m$ such indices; their total contribution is bounded by $CB^2m$, $CB^2m\log^2(m+1)$, or $CB^2m^{3-2\gamma}$ in the three regimes $\gamma>1$, $\gamma=1$ and $1/2<\gamma<1$. For $j\le-m$, put $v=-j\ge m$ and bound the inner sum by $2^\gamma Bm(v+1)^{-\gamma}$; this contributes at most $4^\gamma B^2m^2\sum_{v\ge m}(v+1)^{-2\gamma}\le C_\gamma B^2m^{3-2\gamma}$, which is at most a constant times $m$ for $\gamma>1$ and is absorbed by $m\log^2(m+1)$ for $\gamma=1$. Dividing by $m^2$ and taking the square root gives the claim.
\end{proof}

\begin{lemma}\label{lem:sigmarate}
Under Assumption~\ref{ass:fdm}, $\sigma_m\le C_\sigma\rho_m$ with $\rho_m$ as in \eqref{eq:sigmarate} and $C_\sigma=C_\gamma A$, where $C_\gamma$ is the constant of Proposition~\ref{prop:convolution}.
\end{lemma}

\begin{proof}
By \eqref{eq:couplingidentity}, Assumption~\ref{ass:fdm} gives $\sup_s\pi_j(s)\le2^{-1/2}A(1\vee j)^{-\gamma}\le A(1\vee j)^{-\gamma}$ for all $j\ge0$, and Proposition~\ref{prop:convolution} applies with $B=A$.
\end{proof}

\subsection{The Empirical Quantile of the Oracle Scores}\label{app:twothreshold}

This subsection turns the variance bound of Appendix~\ref{app:projection} into bounds on the tail and the mean of the oracle calibration error $\Qor$ (Proposition~\ref{prop:quantileform}), and fixes the thresholds $m_0$ and $m_1$ and the constants $C_1$ and $C_2$ of Theorems~\ref{thm:coverage}--\ref{thm:cmr}.

We write $d_m=k/m-(1-\alpha)$ for the offset of the rank $k=\lceil(m+1)(1-\alpha)\rceil$ from its target, so that $0\le d_m<2/m$, and we use the thresholds
\begin{equation}\label{eq:m0}
m_0=\max\Big\{\frac{4}{f_\alpha r},\ \frac{2(1-\alpha)}{\alpha}\Big\},
\end{equation}
for which $m\ge m_0$ implies $d_m/f_\alpha\le r/2$ and $\min(k,m-k+1)\ge m\min(1-\alpha,\alpha/2)$, and
\begin{equation}\label{eq:m1}
m_1=\min\Big\{m'\ge m_0:\ \sup_{m\ge m'}t_m\le\frac r2\Big\},\qquad
t_m=\max\Big\{\frac{4}{mf_\alpha},\ \Big(\frac{16C_\sigma^2\rho_m^2}{L_\alpha f_\alpha^2}\Big)^{1/3}\Big\},
\end{equation}
which is finite because $\rho_m\to0$ under Assumption~\ref{ass:fdm}. Indeed $k\ge(m+1)(1-\alpha)\ge m(1-\alpha)$ and $m-k+1\ge\alpha m-(1-\alpha)\ge\alpha m/2$ when $m\ge2(1-\alpha)/\alpha$. The constants of Theorems~\ref{thm:coverage}--\ref{thm:cmr} are
\begin{equation}\label{eq:C1}
C_1=\frac32\Big(\frac{4L_\alpha C_\sigma}{f_\alpha}\Big)^{2/3}+\frac{4L_\alpha}{f_\alpha},
\end{equation}
\begin{equation}\label{eq:C2}
C_2=\frac{4}{f_\alpha}+\frac{4\sqrt2\,C_\sigma}{f_\alpha}\Big(1+\frac{M}{r}\Big),\qquad M=\Big\{\frac{\E(\Sor_0)^2}{\min(1-\alpha,\alpha/2)}\Big\}^{1/2},
\end{equation}
where $\E(\Sor_0)^2$ is $\E(\SQ_0)^2$ for CQR and $\E(\SM_0-q_\alpha)^2$ for CMR.

Write $v_m(s)=\Var(\sum_{i=1}^m\ind{\Sor_{n+i}\le s})=m^2\Var\{\Fhat(s)\}$, so that $\sigma_m^2=\sup_{\abs{s}\le r}v_m(s)/m^2$, and put $F_+(t)=F(t)-(1-\alpha)$ and $F_-(t)=(1-\alpha)-F(-t)$. If $F_+(t)>d_m$ and $F_-(t)+d_m>0$, then
\begin{equation}\label{eq:exacttwothreshold}
\Prob(\abs{\Qor}>t)\le\min\bigg\{1,\ \frac{v_m(t)}{m^2\{F_+(t)-d_m\}^2}+\frac{v_m(-t)}{m^2\{F_-(t)+d_m\}^2}\bigg\}.
\end{equation}
Indeed, $\Qor>t$ implies $\Fhat(t)<k/m$, hence $F(t)-\Fhat(t)>F_+(t)-d_m$; and $\Qor<-t$ implies $\Fhat(-t)\ge k/m$, hence $\Fhat(-t)-F(-t)\ge F_-(t)+d_m$. Applying Chebyshev's inequality to each event and adding the two bounds proves \eqref{eq:exacttwothreshold}.

\begin{proposition}[Quantile form]\label{prop:quantileform}
Under Assumption~\ref{ass:local}(a), put $a_m=d_m/f_\alpha$ and $b_m=\sqrt2\sigma_m/f_\alpha$. For $a_m<t\le r$,
\begin{equation}\label{eq:twothreshold}
\Prob(\abs{\Qor}>t)\le\min\{1,\,b_m^2/(t-a_m)^2\}.
\end{equation}
If in addition $M_m\ge\norm{\Qor}_2$ is a deterministic number and $a_m<r$, then
\begin{equation}\label{eq:rankmoment}
\E\abs{\Qor}\le a_m+2b_m+M_m\min\{1,b_m/(r-a_m)\},
\end{equation}
and a valid choice is
\begin{equation}\label{eq:rankM}
M_m=\bigg\{\frac{m\,\E(\Sor_0)^2}{\min(k,m-k+1)}\bigg\}^{1/2}.
\end{equation}
\end{proposition}

\begin{proof}
For $a_m<t\le r$, Assumption~\ref{ass:local}(a) gives $F_+(t)-d_m\ge f_\alpha(t-a_m)$ and $F_-(t)+d_m\ge f_\alpha(t-a_m)$, and each numerator in \eqref{eq:exacttwothreshold} is at most $m^2\sigma_m^2$. This proves \eqref{eq:twothreshold}. For \eqref{eq:rankmoment}, the tail-integral identity gives
\begin{align*}
\E\min(\abs{\Qor},r)=\int_0^r\Prob(\abs{\Qor}>t)\,dt
&\le a_m+\int_{a_m}^r\min\{1,b_m^2/(t-a_m)^2\}\,dt\\
&\le a_m+\int_0^\infty\min\{1,b_m^2/u^2\}\,du=a_m+2b_m,
\end{align*}
with the integral interpreted as zero when $b_m=0$, and the Cauchy--Schwarz inequality bounds $\E(\abs{\Qor}-r)_+\le\E[\abs{\Qor}\ind{\abs{\Qor}>r}]$ by $M_m\Prob(\abs{\Qor}>r)^{1/2}\le M_m\min\{1,b_m/(r-a_m)\}$. For \eqref{eq:rankM}, if $\Qor\ge0$, then at least $m-k+1$ scores have squares at least $(\Qor)^2$; if $\Qor<0$, then at least $k$ scores do; hence $\min(k,m-k+1)(\Qor)^2\le\sum_{i=1}^m(\Sor_{n+i})^2$, and taking expectations proves the claim.
\end{proof}

At fixed $\alpha$ and $m\ge m_0$, \eqref{eq:rankM} gives $M_m\le M$ with the constant $M$ of \eqref{eq:C2}, a bound that does not depend on $m$; the constants are not uniform as $\alpha\downarrow0$.

\subsection{Proofs of Theorems \ref{thm:coverage}--\ref{thm:cmr}}\label{app:jointproof}

\begin{proof}[Proof of Theorem \ref{thm:length}]\label{app:lengthproof}
For $4/(mf_\alpha)\le t\le r$ we have $d_m<2/m\le f_\alpha t/2$, so $t-a_m\ge t/2$ and \eqref{eq:twothreshold} gives $\Prob(\abs{\Qor}>t)\le4b_m^2/t^2=8\sigma_m^2/(f_\alpha t)^2\le8C_\sigma^2\rho_m^2/(f_\alpha t)^2$, by Lemma~\ref{lem:sigmarate}. On the event $\{E_n\le e,\ \abs{\Qor}\le t\}$, \eqref{eq:comparison} bounds the length error by $4e+2t$, and a union bound proves \eqref{eq:lengthtail}.

For the expectation, $m\ge m_0$ gives $a_m\le r/2$, so that $\min\{1,b_m/(r-a_m)\}\le2b_m/r$, and \eqref{eq:rankmoment} with \eqref{eq:rankM} and $M_m\le M$ gives
\begin{equation}\label{eq:rankexpectation}
\E\abs{\Qor}\le\frac{2}{mf_\alpha}+\frac{2\sqrt2\,\sigma_m}{f_\alpha}\Big(1+\frac{M}{r}\Big)\le\frac{C_2}{2}\rho_m,
\end{equation}
where the second inequality uses $\sigma_m\le C_\sigma\rho_m$ and $1/m\le\rho_m$, which holds because $\rho_m\ge m^{-1/2}$ for $m\ge2$ and $m_0\ge4$. Taking expectations in \eqref{eq:comparison} proves \eqref{eq:lengthexpectation}.
\end{proof}

\begin{proof}[Proof of Theorem \ref{thm:coverage}]\label{app:coverage}
Fix $e\ge0$ and $t$ with $4/(mf_\alpha)\le t$ and $2e+t\le r$, and put $\eta=8C_\sigma^2\rho_m^2/(f_\alpha t)^2$, so that $\Prob(\abs{\Qor}>t)\le\eta$ by the first step of the proof of Theorem~\ref{thm:length}. On the event $\{E_n\le e,\ \abs{\Qor}\le t\}$, the inequalities \eqref{eq:centeredcomparison} imply
\[
\{\Sor_{n+m+1}\le-2e-t\}\subseteq\{\Shat_{n+m+1}\le\Shat_{(k)}\}\subseteq\{\Sor_{n+m+1}\le2e+t\}.
\]
For example, if $\Sor_{n+m+1}\le-2e-t$, then $\Shat_{n+m+1}\le-e-t\le\Shat_{(k)}$, since $\Shat_{(k)}\ge\Qor-e\ge-t-e$; the other inclusion follows from $\Sor_{n+m+1}\le\Shat_{n+m+1}+e\le\Shat_{(k)}+e\le t+2e$. Removing the exceptional event at a cost of at most $\Prob(E_n>e)+\eta$ gives
\[
F(-2e-t)-\Prob(E_n>e)-\eta\le\Prob\{Y_{n+m+1}\in\ChatQ(X_{n+m+1})\}\le F(2e+t)+\Prob(E_n>e)+\eta,
\]
and the upper bounds of Assumption~\ref{ass:local}(b), valid because $2e+t\le r$, give
\[
\big|\Prob\{Y_{n+m+1}\in\ChatQ(X_{n+m+1})\}-(1-\alpha)\big|\le L_\alpha(2e+t)+\Prob(E_n>e)+\frac{8C_\sigma^2\rho_m^2}{f_\alpha^2t^2}.
\]
Now take $e=e_n\le r/4$ and $t=t_m$ as in \eqref{eq:m1}; for $m\ge m_1$ we have $t_m\le r/2$, so that $2e_n+t_m\le r$, and $t_m\ge4/(mf_\alpha)$ by construction. Writing $t^*$ for the second term in the maximum defining $t_m$, $8C_\sigma^2\rho_m^2/(f_\alpha^2t^{*2})=L_\alpha t^*/2$, and since $t_m\ge t^*$,
\[
L_\alpha t_m+\frac{8C_\sigma^2\rho_m^2}{f_\alpha^2t_m^2}\le L_\alpha t^*+\frac{4L_\alpha}{mf_\alpha}+\frac{L_\alpha t^*}{2}=\frac32\Big(\frac{4L_\alpha C_\sigma\rho_m}{f_\alpha}\Big)^{2/3}+\frac{4L_\alpha}{mf_\alpha}.
\]
With $1/m\le\rho_m^{2/3}$, valid for $m\ge2$, this proves \eqref{eq:coveragerate} with the constant $C_1$ of \eqref{eq:C1}. The exponent $2/3$ comes from the window: the coverage error is at most $L_\alpha t$ on a window of half-width $t$ around the threshold, plus the probability of order $\rho_m^2/t^2$ that the calibrated threshold leaves the window, and $t\asymp\rho_m^{2/3}$ balances the two terms; when the score at the test time has a bounded density given the past, no window is needed and \eqref{eq:coveragepred} retains the full rate.

For \eqref{eq:coveragepred}, let $D=2E_n+\abs{\Qor}$, a nonnegative random variable measurable with respect to $\Fcal_{n+m}$ and the training randomness, which is independent of the series, so that the conditional law of $\Sor_{n+m+1}$ given both has the same density bound $L_{\mathrm{pred}}$. The same pathwise inclusions hold with $D$ in place of $2e+t$ on every realization. Denote $A=\{\Shat_{n+m+1}\le\Shat_{(k)}\}$ and $B=\{\Sor_{n+m+1}\le0\}$. Then $A\setminus B\subseteq\{0<\Sor_{n+m+1}\le D\}$ and $B\setminus A\subseteq\{-D<\Sor_{n+m+1}\le0\}$, and the conditional density bound gives probability at most $L_{\mathrm{pred}}\E D$ for each. Since $\abs{\Prob(A)-\Prob(B)}=\abs{\Prob(A\setminus B)-\Prob(B\setminus A)}$ is at most the larger of the two nonnegative terms, it is at most $L_{\mathrm{pred}}\E D$, and $\Prob(B)=F(0)=1-\alpha$. The bound \eqref{eq:rankexpectation} completes the proof.
\end{proof}

The predictive-density condition of \eqref{eq:coveragepred} holds, for example, when $X_{n+m+1}$ is $\Fcal_{n+m}$-measurable and $Y_{n+m+1}=a+\sigma\epsilon$, where $a$ and $\sigma\ge\sigma_0>0$ are $\Fcal_{n+m}$-measurable and $\epsilon$ is independent of the past with density at most $M_\epsilon$: conditional on the past, both oracle CQR endpoints are fixed and the score density is at most $2M_\epsilon/\sigma_0$, so one can take $L_{\mathrm{pred}}=2M_\epsilon/\sigma_0$. The condition concerns the predictive law and is stronger than a density bound for the stationary marginal law; a process with a positive probability of repeating its preceding value does not satisfy it, and for such a process only \eqref{eq:coveragerate} applies.

\begin{proof}[Proof of Theorem \ref{thm:cmr}]
With the centered scores of Appendix~\ref{app:projection}, \eqref{eq:centeredcomparison} gives $\abs{\Shat_{(k)}-\Qor}\le E_n$ and $\abs{\ShatM_{(k)}-q_\alpha}=\abs{\Shat_{(k)}}\le E_n+\abs{\Qor}$, and Assumptions~\ref{ass:local} and \ref{ass:fdm} for CMR are the assumptions used above for the centered score $\Sor_t=\SM_t-q_\alpha$. The proofs of Theorems~\ref{thm:coverage} and \ref{thm:length} then apply verbatim: the coverage argument with the event $\{Y_{n+m+1}\in\ChatM(X_{n+m+1})\}=\{\Shat_{n+m+1}\le\Shat_{(k)}\}$ and $D=2E_n+\abs{\Qor}$ as before, which proves \eqref{eq:cmrcoverage} and \eqref{eq:cmrcoveragepred}, and the length argument with $2E_n+2\abs{\Qor}$ in place of $4E_n+2\abs{\Qor}$, which proves \eqref{eq:cmrlengthtail} and \eqref{eq:cmrlength} with \eqref{eq:rankexpectation}.
\end{proof}

The proofs use the test observation only through its oracle score: through the marginal distribution function $F$, which by stationarity is the same at every time, and, for \eqref{eq:coveragepred} and \eqref{eq:cmrcoveragepred}, through the density bound of its conditional law given $\Fcal_{n+m}$. For a fixed horizon $h\ge1$, the conditional law of $\Sor_{n+m+h}$ given $\Fcal_{n+m}$ has the density $\E\{p_{n+m+h}(\cdot)\mid\Fcal_{n+m}\}$, where $p_{n+m+h}$ is the conditional density of $\Sor_{n+m+h}$ given $\Fcal_{n+m+h-1}$ and is at most $L_{\mathrm{pred}}$ by stationarity. The coverage bounds of Theorems~\ref{thm:coverage} and~\ref{thm:cmr} therefore hold, with the same constants, for $Y_{n+m+h}$ and the interval evaluated at $X_{n+m+h}$.

The convergence statements after Theorem~\ref{thm:length} are the following: under Assumptions~\ref{ass:train}--\ref{ass:fdm}, as $n\to\infty$ and $m\to\infty$ in any manner,
\begin{equation}\label{eq:joint}
\Prob\{Y_{n+m+1}\in\ChatQ(X_{n+m+1})\}\to1-\alpha\quad\text{and}\quad\sup_x\abs{\len\ChatQ(x)-\len\CQ(x)}\overset{p}{\to}0,
\end{equation}
and for CMR the coverage tends to $1-\alpha$ and $\sup_x\abs{\len\ChatM(x)-\len\CM(x)}=2\abs{\ShatM_{(k)}-q_\alpha}\to0$ in probability. For \eqref{eq:joint}, since $\gamma>1/2$, $\rho_m\to0$, and $e_n\to0$, so that $e_n\le r/4$ and $m\ge m_1$ for all large $n$ and $m$, and the right-hand side of \eqref{eq:coveragerate} tends to zero; for the length, \eqref{eq:lengthtail} with $e=e_n$ and any fixed $t\in(0,r]$ has a right-hand side that tends to zero and a threshold $4e_n+2t\to2t$, and $t$ is arbitrary. The same argument applied to \eqref{eq:cmrcoverage} and \eqref{eq:cmrlengthtail} gives the statements for CMR. Neither statement relates the growth of $n$ to that of $m$.

\section{The Assumptions: Verification and Examples}\label{app:assumptions}

This appendix verifies Assumption~\ref{ass:train} for linear quantile regression and extends Theorems~\ref{thm:coverage}--\ref{thm:cmr} to learners whose fitted endpoints converge to a limit other than the oracle endpoints, verifies Assumption~\ref{ass:local} from a density condition on the response and Assumption~\ref{ass:fdm} from the functional dependence measure of the scores, then verifies the four examples of Section~\ref{sec:main} and compares the resulting guarantees with those formulated through mixing.

\subsection{Assumption \ref{ass:train} and the Learner}\label{app:training}

Assumption~\ref{ass:train} concerns the learner only. We verify it for the parametric quantile-regression estimator computed from all $n$ training observations, without any rate of decay of the dependence: the only property of the series that is used is ergodicity, which a stationary causal functional $Z_t=G(\ldots,\varepsilon_{t-1},\varepsilon_t)$ of \iid\ innovations always has, because the shift on the innovation sequence is a Bernoulli shift, hence ergodic, and the process is a factor of it \citep{cornfeld1982ergodic}.

Let $\varrho_\tau(v)=v(\tau-\ind{v<0})$ be the pinball loss at level $\tau\in(0,1)$, let $R_\tau(\theta)=\E\varrho_\tau(Y_t-X_t^\tr\theta)$ be the population risk, and let $\widehat R_{n,\tau}(\theta)=n^{-1}\sum_{t=1}^n\varrho_\tau(Y_t-X_t^\tr\theta)$ be the training risk.

\begin{lemma}[Consistency of quantile regression on an ergodic series]\label{lem:ergodicqr}
Suppose that $\norm{X_t}\le B_x$ almost surely, $\E\abs{Y_t}<\infty$, and that for $\tau\in\{\alpha/2,1-\alpha/2\}$ the conditional $\tau$-quantile of $Y_t$ given $X_t=x$ is $x^\tr\theta^*_\tau$ on the covariate domain, where $\theta^*_\tau$ is the unique minimizer of $R_\tau$ over a compact set $\Theta\subset\R^p$. Let $\widehat\theta_\tau\in\Theta$ satisfy $\widehat R_{n,\tau}(\widehat\theta_\tau)\le\inf_{\theta\in\Theta}\widehat R_{n,\tau}(\theta)+o_n$ with $o_n\to0$ in probability. Then $\widehat\theta_\tau\to\theta^*_\tau$ in probability, and the fitted endpoints $\lhat(x)=x^\tr\widehat\theta_{\alpha/2}$ and $\uhat(x)=x^\tr\widehat\theta_{1-\alpha/2}$ satisfy $\EQ\le B_x\max_\tau\norm{\widehat\theta_\tau-\theta^*_\tau}\to0$ in probability, so that Assumption~\ref{ass:train} holds.
\end{lemma}

\begin{proof}
Fix $\tau$. Since $\abs{\varrho_\tau(v)-\varrho_\tau(v')}\le\abs{v-v'}$, the map $\theta\mapsto\varrho_\tau(Y_t-X_t^\tr\theta)$ is $B_x$-Lipschitz on $\Theta$ for every realization, so $R_\tau$ and $\widehat R_{n,\tau}$ are $B_x$-Lipschitz, and $R_\tau$ is finite because $\E\abs{Y_t}<\infty$ and $\Theta$ is bounded. For each fixed $\theta$, the ergodic theorem gives $\widehat R_{n,\tau}(\theta)\to R_\tau(\theta)$ almost surely. Given $\epsilon>0$, let $\theta_1,\ldots,\theta_J$ be a finite $\epsilon/(3B_x)$-net of $\Theta$; on the almost sure event on which the convergence holds at every $\theta_j$, the Lipschitz property gives $\sup_{\theta\in\Theta}\abs{\widehat R_{n,\tau}(\theta)-R_\tau(\theta)}\le\epsilon$ for all large $n$. Hence $\sup_\Theta\abs{\widehat R_{n,\tau}-R_\tau}\to0$ almost surely.

Now $R_\tau(\widehat\theta_\tau)\le\widehat R_{n,\tau}(\widehat\theta_\tau)+\sup_\Theta\abs{\widehat R_{n,\tau}-R_\tau}\le\widehat R_{n,\tau}(\theta^*_\tau)+o_n+\sup_\Theta\abs{\widehat R_{n,\tau}-R_\tau}\le R_\tau(\theta^*_\tau)+o_n+2\sup_\Theta\abs{\widehat R_{n,\tau}-R_\tau}$, so $R_\tau(\widehat\theta_\tau)-R_\tau(\theta^*_\tau)\to0$ in probability. Since $R_\tau$ is continuous on the compact set $\Theta$ and $\theta^*_\tau$ is its unique minimizer, for every $\eta>0$ the number $\inf\{R_\tau(\theta)-R_\tau(\theta^*_\tau):\theta\in\Theta,\ \norm{\theta-\theta^*_\tau}\ge\eta\}$ is positive, and therefore $\Prob(\norm{\widehat\theta_\tau-\theta^*_\tau}\ge\eta)\to0$. Finally, on the covariate domain, $\abs{x^\tr\widehat\theta_\tau-x^\tr\theta^*_\tau}\le B_x\norm{\widehat\theta_\tau-\theta^*_\tau}$, which gives the bound on $\EQ$; the sequences $e_n$ and $\xi_n$ of Assumption~\ref{ass:train} can be taken as $e_n=\eta_n$ and $\xi_n=\Prob(B_x\max_\tau\norm{\widehat\theta_\tau-\theta^*_\tau}>\eta_n)$ for any $\eta_n\downarrow0$ slowly enough that $\xi_n\to0$.
\end{proof}

Uniqueness of the minimizer holds, for example, when the conditional density of $Y_t$ given $X_t$ is positive at $x^\tr\theta^*_\tau$ and $\E X_tX_t^\tr$ is positive definite.

\subsection{Misspecified Learners}\label{app:misspecified}

This subsection extends Theorems~\ref{thm:coverage}--\ref{thm:cmr} to a learner whose fitted endpoints converge to a limit other than the oracle endpoints, as those of linear quantile regression do when the conditional quantiles are not linear in the covariates. The calibrated interval then keeps its coverage and approaches the limit interval adjusted to marginal coverage $1-\alpha$, whereas the fitted interval itself misses the nominal level unless its limit already has marginal coverage $1-\alpha$.

Let $\ell_\infty$ and $u_\infty$ be measurable functions on the covariate domain, let $s_\infty(x,y)=\max\{\ell_\infty(x)-y,\ y-u_\infty(x)\}$, let $q_\infty$ be the $(1-\alpha)$-quantile of $s_\infty(X_t,Y_t)$ under the stationary law, and put
\[
C_\infty(x)=[\ell_\infty(x)-q_\infty,\ u_\infty(x)+q_\infty],\qquad E^\infty_n=\max\{\norm{\lhat-\ell_\infty}_\infty,\norm{\uhat-u_\infty}_\infty\}.
\]
Since $Y_t$ lies in $[\ell_\infty(X_t)-c,\ u_\infty(X_t)+c]$ exactly when $s_\infty(X_t,Y_t)\le c$, the interval $C_\infty(x)$ is the shortest interval of this form with marginal coverage at least $1-\alpha$; for the oracle endpoints, $q_\infty=0$ and $C_\infty=\CQ$.

\begin{proposition}[Misspecified learners]\label{prop:misspecified}
Suppose that $u_\infty(x)-\ell_\infty(x)+2q_\infty\ge0$ for every $x$ in the covariate domain, that Assumption~\ref{ass:train} holds with $E^\infty_n$ in place of $E_n$, and that Assumptions~\ref{ass:local} and~\ref{ass:fdm} hold for the centered limit score $\Sor_t=s_\infty(X_t,Y_t)-q_\infty$. Then Theorems~\ref{thm:coverage} and~\ref{thm:length} hold with $\Sor_t$, $E^\infty_n$ and $C_\infty$ in place of $\SQ_t$, $\EQ$ and $\CQ$. Moreover, if $q_\infty\ne0$, then for all $n$ with $e_n\le\min(\abs{q_\infty},r)/2$ the coverage of the fitted interval satisfies
\[
\big|\Prob\{\lhat(X_{n+m+1})\le Y_{n+m+1}\le\uhat(X_{n+m+1})\}-(1-\alpha)\big|\ge\frac{f_\alpha\min(\abs{q_\infty},r)}{2}-\xi_n.
\]
\end{proposition}

\begin{proof}
Put $\Shat_t=\ShatQ_t-q_\infty$. The inequality for maxima in the proof of \eqref{eq:comparison} gives $\abs{\Shat_t-\Sor_t}\le E^\infty_n$ for $n<t\le n+m+1$, so that \eqref{eq:centeredcomparison} holds with $E^\infty_n$ in place of $E_n$. The distribution function $F$ of $\Sor_t$ satisfies $F(0)=1-\alpha$, by the definition of $q_\infty$ and the continuity of $F$ at $0$ required by Assumption~\ref{ass:local}. The coverage event is $\{\Shat_{n+m+1}\le\Shat_{(k)}\}$, as before, because subtracting $q_\infty$ from every fitted score does not change the calibrated interval. The proof of Theorem~\ref{thm:coverage} uses the oracle endpoints only through these three facts, and it applies verbatim. For the length, put $\widehat w(x)=\uhat(x)-\lhat(x)$ and $w_\infty(x)=u_\infty(x)-\ell_\infty(x)$, so that $\len\ChatQ(x)=[\widehat w(x)+2q_\infty+2\Shat_{(k)}]_+$ and $\len C_\infty(x)=w_\infty(x)+2q_\infty\ge0$. Since the positive-part map is $1$-Lipschitz,
\[
\abs{\len\ChatQ(x)-\len C_\infty(x)}\le\abs{\widehat w(x)-w_\infty(x)}+2\abs{\Shat_{(k)}}\le2E^\infty_n+2\big(E^\infty_n+\abs{\Qor}\big),
\]
which is the length bound of \eqref{eq:comparison} with $E^\infty_n$ and $C_\infty$ in the roles of $\EQ$ and $\CQ$, and the proof of Theorem~\ref{thm:length} applies verbatim.

For the fitted interval, $Y_{n+m+1}$ lies between $\lhat(X_{n+m+1})$ and $\uhat(X_{n+m+1})$ exactly when $\ShatQ_{n+m+1}\le0$, that is, when $\Shat_{n+m+1}\le-q_\infty$. Put $\delta=\min(\abs{q_\infty},r)/2$, so that $e_n\le\delta$. If $q_\infty>0$, this event and $\{E^\infty_n\le e_n\}$ imply $\Sor_{n+m+1}\le-q_\infty+e_n\le-\delta$, so that the coverage is at most $F(-\delta)+\xi_n\le1-\alpha-f_\alpha\delta+\xi_n$ by Assumption~\ref{ass:local}(a). If $q_\infty<0$, the events $\{\Sor_{n+m+1}\le\abs{q_\infty}-e_n\}$ and $\{E^\infty_n\le e_n\}$ together imply it, and $\abs{q_\infty}-e_n\ge\delta$, so that the coverage is at least $F(\delta)-\xi_n\ge1-\alpha+f_\alpha\delta-\xi_n$.
\end{proof}

The same argument applies to CMR: for a limit center $\mu_\infty$, the $(1-\alpha)$-quantile $q_\infty$ of $\abs{Y_t-\mu_\infty(X_t)}$ and $E^\infty_n=\norm{\muhat-\mu_\infty}_\infty$, Theorem~\ref{thm:cmr} holds with $[\mu_\infty(x)-q_\infty,\ \mu_\infty(x)+q_\infty]$ in place of $\CM(x)$, provided that Assumption~\ref{ass:train} holds for $E^\infty_n$ and Assumptions~\ref{ass:local} and~\ref{ass:fdm} hold for the centered score $\abs{Y_t-\mu_\infty(X_t)}-q_\infty$.

For linear quantile regression, the proof of Lemma~\ref{lem:ergodicqr} uses correct specification only in its last step: without it, and with $\theta^*_\tau$ the unique minimizer of $R_\tau$ over $\Theta$ as before, $\widehat\theta_\tau\to\theta^*_\tau$ in probability, and $E^\infty_n\le B_x\max_\tau\norm{\widehat\theta_\tau-\theta^*_\tau}\to0$ in probability for the limit pair $\ell_\infty(x)=x^\tr\theta^*_{\alpha/2}$ and $u_\infty(x)=x^\tr\theta^*_{1-\alpha/2}$. Assumption~\ref{ass:train} therefore holds for $E^\infty_n$ whether or not the conditional quantiles are linear in the covariates, and Assumptions~\ref{ass:local} and~\ref{ass:fdm} for the limit score are verified as in Appendices~\ref{app:margin} and~\ref{app:smoothingproof}, with the endpoints of $C_\infty(x)$ in place of the oracle endpoints. The length of the calibrated interval thus approaches that of $C_\infty$. When the covariates include a constant and $\theta^*_\tau$ is an interior point of $\Theta$, the first-order condition for the intercept gives $\Prob(Y_t\le X_t^\tr\theta^*_\tau)=\tau$ for continuous conditional laws, so that $q_\infty=0$ if $\ell_\infty\le u_\infty$: the fitted interval then attains the nominal level in the limit, and calibration adds the non-asymptotic coverage bound of Theorem~\ref{thm:coverage}. When $q_\infty\ne0$, as can happen without a constant, calibration restores the nominal level that the fitted interval misses.

\subsection{Assumption \ref{ass:local} from the Response Density}\label{app:margin}

This subsection derives Assumption~\ref{ass:local} for CQR from bounds on the conditional density of the response near the oracle endpoints. Assume $u_\alpha(x)-\ell_\alpha(x)\ge2r$. For $\abs{s}\le r$,
\[
F^Q(s)=\E\big[F_{Y\mid X}(u_\alpha(X)+s\mid X)-F_{Y\mid X}(\ell_\alpha(X)-s\mid X)\big].
\]
If the conditional density lies between $f_0$ and $L_0$ on the two radius-$r$ endpoint neighborhoods, then for $0\le t\le r$,
\[
F^Q(t)-F^Q(0)=\E\bigg[\int_{u_\alpha(X)}^{u_\alpha(X)+t}f_{Y\mid X}(v\mid X)\,dv+\int_{\ell_\alpha(X)-t}^{\ell_\alpha(X)}f_{Y\mid X}(v\mid X)\,dv\bigg],
\]
and each integral lies between $f_0t$ and $L_0t$. The same reasoning applies to $F^Q(0)-F^Q(-t)$. Hence Assumption~\ref{ass:local} holds with $f_\alpha=2f_0$ and $L_\alpha=2L_0$. If the conditional density is at most $L_0$ everywhere, then $F^Q(s)$ is, for every $s$, the expectation of the positive part of $F_{Y\mid X}(u_\alpha(X)+s\mid X)-F_{Y\mid X}(\ell_\alpha(X)-s\mid X)$, a nondecreasing function of $s$ with derivative at most $2L_0$, so that the marginal density of $\SQ_t$ is at most $L=2L_0$, the density bound of Proposition~\ref{prop:sufficient}(a).

\subsection{Assumption \ref{ass:fdm} from the Functional Dependence Measure}\label{app:smoothingproof}

This subsection gives two sufficient conditions for Assumption~\ref{ass:fdm} in terms of the functional dependence measure of the scores (Proposition~\ref{prop:sufficient}) and bounds this measure through the dependence of the observations. In the notation of Appendix~\ref{app:projection}, $\omega(j)=\sup_s\norm{W_j(s)-W_j^*(s)}_2$ with $W_j(s)-W_j^*(s)=\E\{J_j(s)-J_j^*(s)\mid\Fcal_0,\varepsilon_0'\}$, and $\delta_q(j)=\norm{\Sor_j-\Sor_j^*}_q$ is the functional dependence measure of the score.

\begin{proposition}[Sufficient conditions]\label{prop:sufficient}
Let $A'>0$ and $\gamma'>0$.
\begin{enumerate}
\item[(a)] If the marginal density of $\Sor_t$ is at most $L$ and $\delta_q(j)\le A'(1\vee j)^{-\gamma'}$ for all $j\ge0$ and some $q\ge1$, then $\omega(j)\le\sqrt3(LA')^{q/(2q+2)}(1\vee j)^{-\gamma'q/(2q+2)}$ for all $j\ge0$, so that Assumption~\ref{ass:fdm} holds with $\gamma=\gamma'q/(2q+2)$ whenever $\gamma'q>q+1$.
\item[(b)] Suppose $\Sor_t=H(V_t,\varepsilon_t)$, where $V_t$ is $\Fcal_{t-1}$-measurable, and that $K_s(v)=\Prob\{H(v,\varepsilon_t)\le s\}$ is $L_H$-Lipschitz in $v$ for every $s$. Then $\norm{W_j(s)-W_j^*(s)}_2\le L_H\norm{V_j-V_j^*}_2$ for every $j\ge1$ and every $s$. Consequently, if $\norm{V_j-V_j^*}_2\le A'(1\vee j)^{-\gamma'}$ for all $j\ge1$ and $\gamma'>1/2$, then Assumption~\ref{ass:fdm} holds with $\gamma=\gamma'$ and $A=\max(2^{-1/2},L_HA')$.
\end{enumerate}
\end{proposition}

\begin{proof}
Part (a) is \eqref{eq:projectionbound}. For part (b), fix $j\ge1$ and $s$. Conditional on $\Fcal_{j-1}$ and $\varepsilon_0'$, the common current innovation $\varepsilon_j$ is fresh, so that $\E\{J_j(s)\mid\Fcal_{j-1},\varepsilon_0'\}=K_s(V_j)$ and $\E\{J_j^*(s)\mid\Fcal_{j-1},\varepsilon_0'\}=K_s(V_j^*)$. Since $j\ge1$, the tower property gives
\[
W_j(s)-W_j^*(s)=\E\{K_s(V_j)-K_s(V_j^*)\mid\Fcal_0,\varepsilon_0'\},
\]
and conditional Jensen's inequality together with the Lipschitz assumption gives $\norm{W_j(s)-W_j^*(s)}_2\le L_H\norm{V_j-V_j^*}_2$. The requirement $j\ge1$ ensures that the innovation being integrated out is not the innovation being replaced; at $j=0$, $W_0(s)=J_0(s)$ and $W_0^*(s)=J_0^*(s)$ are independent and identically distributed given $\Fcal_{-1}$, so that $\norm{W_0(s)-W_0^*(s)}_2^2=2\E\{p(1-p)\}\le1/2$ with $p=\Prob(\Sor_0\le s\mid\Fcal_{-1})$, which is the reason for the constant $A=\max(2^{-1/2},L_HA')$.
\end{proof}

Part (a) requires only a bounded marginal density of the score, at the price of the power $q/(2q+2)$; since $\gamma=\gamma'q/(2q+2)>1/2$ forces $\gamma'>1+1/q$, it does not admit long memory, and it is the route for series whose conditional next-step law has atoms, such as the Bernoulli autoregression of Appendix~\ref{app:examples}. Part (b) applies when the conditional distribution function of the score given the past, $K_s(V_t)$, is a Lipschitz function of the part $V_t$ that the past determines; the decay exponent of $\norm{V_j-V_j^*}_2$ then transfers to $\omega(j)$ without loss. This is the structure of a location model $Y_t=m_t+a_0\varepsilon_t$ and of a scale model $Y_t=\sigma_t\varepsilon_t$ with past-determined $m_t$ or $\sigma_t$ (Appendix~\ref{app:examples}), for which $\norm{V_j-V_j^*}_2$ is bounded by the functional dependence measure of $m_t$ or $\sigma_t$ plus a Lipschitz constant times that of $X_t$; for a linear process the exponent is the memory exponent $\beta$, so part (b) admits long memory.

When the oracle endpoints are Lipschitz, the dependence measure of the scores is controlled by that of the observations. For the CQR score, the inequality for maxima used in the proof of \eqref{eq:comparison} gives
\[
\abs{\SQ_t-\SQs_t}\le\abs{Y_t-Y_t^*}+\max\{\abs{\ell_\alpha(X_t)-\ell_\alpha(X_t^*)},\abs{u_\alpha(X_t)-u_\alpha(X_t^*)}\},
\]
and Minkowski's inequality proves $\delta^Q_q(j)\le\norm{Y_j-Y_j^*}_q+K_0\norm{X_j-X_j^*}_q$ for $K_0$-Lipschitz endpoints; for CMR the corresponding constant is the Lipschitz constant of $\mu$. For a stationary recursion $Z_t=H(Z_{t-1},\varepsilon_t)$, a sufficient contraction condition is $\norm{H(z,\varepsilon_t)-H(z',\varepsilon_t)}_q\le\lambda_H\norm{z-z'}$ for deterministic $z,z'$ and some $0<\lambda_H<1$, with the norm inside the expectation interpreted in the state space. If a stationary $L^q$ solution exists, iterating the coupled recursion gives $\delta^Z_q(j)\le\lambda_H^j\delta^Z_q(0)$ for the observations; this is the geometric moment-contraction route of \citet{wu2004limit}, and the contraction has to be checked for the particular recursion.

\subsection{Verification of the Examples of Section \ref{sec:main}}\label{app:examples}

This subsection verifies the assumptions for the four examples of Section~\ref{sec:main}.

\paragraph{A linear process with discrete innovations.}
Let $B_t$ be \iid\ Bernoulli$(1/2)$ and $W_t=\sum_{j\ge0}2^{-j}B_{t-j}=W_{t-1}/2+B_t$, which is uniform on $[0,2)$ by the binary expansion of $W_t/2$. Changing $B_0$ changes $W_j$ by exactly $2^{-j}(B_0-B_0')$, so $\norm{W_j-W_j^*}_q=2^{-j}\norm{B_0-B_0'}_q\le2^{-j}$. To verify the failure of strong mixing directly, let $\alpha_{\mathrm{mix}}(h)$ be the supremum of $\abs{\Prob(A\cap D)-\Prob(A)\Prob(D)}$ over $A\in\sigma(W_s:s\le0)$ and $D\in\sigma(W_s:s\ge h)$. The recursion gives $2^hW_h=2\sum_{j=1}^{h}2^{j-1}B_j+W_0$, so $W_0$ is the remainder of $2^hW_h$ modulo $2$ and is a measurable function of $W_h$; taking $A=\{W_0\le1\}$ and $D$ the same event expressed through $W_h$ gives $\alpha_{\mathrm{mix}}(h)\ge1/4$ for every $h$, which is the argument of \citet{andrews1984nonstrong}. In the location model $Y_t=\mu+W_t$ with a constant covariate, the conditional median is $\mu+1$, the oracle CMR score $\SM_t=\abs{W_t-1}$ is uniform on $[0,1]$ with $q_\alpha=1-\alpha$, so the centered distribution function $F^M$ has slope one on $[-r,r]$, and Assumption~\ref{ass:local} holds with $f_\alpha=L_\alpha=1$ and $r=\min(\alpha,1-\alpha)$; the oracle CQR endpoints are $\mu+\alpha$ and $\mu+2-\alpha$, so that the CQR score $\SQ_t=\abs{W_t-1}-(1-\alpha)$ coincides with the centered CMR score, and what follows applies to CQR as well. The score has marginal density one, and $\delta^M_q(j)\le\norm{W_j-W_j^*}_q\le2^{-j}$ because the absolute value is $1$-Lipschitz, so Proposition~\ref{prop:sufficient}(a) gives $\omega^M(j)\le\sqrt3\,2^{-jq/(2q+2)}$, and Assumption~\ref{ass:fdm} holds with every $\gamma>1/2$, so that $\rho_m=m^{-1/2}$. The empirical median, and for CQR the empirical $\alpha/2$- and $(1-\alpha/2)$-quantiles, of the training block are consistent by ergodicity (Lemma~\ref{lem:ergodicqr} with a constant covariate), so Assumption~\ref{ass:train} holds for these learners.

\paragraph{Gaussian long memory.}
Under \eqref{eq:linearprocess}, the oracle CMR score in $Y_t=\mu+G_t$ is $\SM_t=\abs{G_t}$, with $q_\alpha=z_\alpha$; the centered score $\Sor_t=\abs{G_t}-z_\alpha$ has distribution function $F^M(s)=2\Phi(z_\alpha+s)-1$ and density $2\phi(z_\alpha+s)$, and on $\abs{s}\le r=z_\alpha/2$ the density lies between $2\phi(3z_\alpha/2)$ and $2\phi(z_\alpha/2)$, which gives Assumption~\ref{ass:local}. Write $\Sor_t=H(V_t,\varepsilon_t)$ with $V_t=\sum_{j\ge1}a_j\varepsilon_{t-j}$ and $H(v,\varepsilon)=\abs{v+a_0\varepsilon}-z_\alpha$. Then $K_s(v)=\Phi\{(z_\alpha+s-v)/\abs{a_0}\}-\Phi\{(-z_\alpha-s-v)/\abs{a_0}\}$ is Lipschitz in $v$ with constant $2\phi(0)/\abs{a_0}$, and $\norm{V_j-V_j^*}_2=\abs{a_j}\,\norm{\varepsilon_0-\varepsilon_0'}_2=\sqrt2\abs{a_j}\le\sqrt2A(1+j)^{-\beta}$, where $A$ is the coefficient bound of \eqref{eq:linearprocess}, so Proposition~\ref{prop:sufficient}(b) gives Assumption~\ref{ass:fdm} with $\gamma=\beta>1/2$ and the constant $\max\{2^{-1/2},2\sqrt2\phi(0)A/\abs{a_0}\}$ in the role of the constant of Assumption~\ref{ass:fdm}. The same holds for the CQR score with the constant oracle endpoints $\mu\mp z_\alpha$, which coincides with the centered CMR score. The functional dependence of the observations themselves is $\norm{G_j-G_j^*}_q=\abs{a_j}\,\norm{\varepsilon_0-\varepsilon_0'}_q$ for every finite $q$, and the failure of strong mixing is proved in Appendix~\ref{app:switch}. For the sample mean, $\EM=\abs{\Gbar}$ with $\Gbar=n^{-1}\sum_{t\le n}G_t$ and $\E\abs{\Gbar}=(2v_n/\pi)^{1/2}\le Cn^{1/2-\beta}$, where $v_n=\Var(\Gbar)$, by \eqref{eq:rhobound} and \eqref{eq:sums}, and Markov's inequality gives Assumption~\ref{ass:train} with, for example, $e_n=n^{(1/2-\beta)/2}$ and $\xi_n=Cn^{(1/2-\beta)/2}$.

\paragraph{Long memory in the scale.}
Let $Y_t=\sigma_t\varepsilon_t$ with \iid\ innovations $\varepsilon_t$ whose law is symmetric with $\E\varepsilon_t^2=1$ and with a bounded, continuous and positive density $f_\varepsilon$ satisfying $\sup_{x\ge0}xf_\varepsilon(x)<\infty$, as for the normal law, and let the volatility be $\sigma_t=g(h_t)$ with $h_t=\sum_{j\ge1}a_j\varepsilon_{t-j}$ a linear process of the past innovations whose coefficients satisfy $\abs{a_j}\le A(1+j)^{-\beta}$ for some $1/2<\beta<1$, and $g$ an $L_g$-Lipschitz function with values in $[\sigma_{\min},\sigma_{\max}]$, $\sigma_{\min}>0$. With a constant covariate the conditional median of $Y_t$ is $0$, and the oracle CMR score is $\SM_t=\sigma_t\abs{\varepsilon_t}$; its marginal density at $s$ is $\E\{f_{\abs{\varepsilon}}(s/\sigma_t)/\sigma_t\}$, which, for $\abs{s-q_\alpha}\le r$ with some $r<q_\alpha$, lies between $\sigma_{\max}^{-1}$ and $\sigma_{\min}^{-1}$ times the infimum and the supremum of $f_{\abs{\varepsilon}}$ over the compact set of arguments $s/\sigma$ with $\abs{s-q_\alpha}\le r$ and $\sigma\in[\sigma_{\min},\sigma_{\max}]$, so that Assumption~\ref{ass:local} holds. The score has the form $H(V_t,\varepsilon_t)=V_t\abs{\varepsilon_t}$ of Proposition~\ref{prop:sufficient}(b) with $V_t=\sigma_t$, and $K_s(v)=F_{\abs{\varepsilon}}(s/v)$ has $\abs{\partial_vK_s(v)}=(s/v^2)f_{\abs{\varepsilon}}(s/v)\le\sigma_{\min}^{-1}\sup_{x\ge0}xf_{\abs{\varepsilon}}(x)=L_H$ uniformly in $s$ for $v\ge\sigma_{\min}$, while $\norm{V_j-V_j^*}_2\le L_g\norm{h_j-h_j^*}_2=\sqrt2L_g\abs{a_j}\le\sqrt2L_gA(1+j)^{-\beta}$. Hence Assumption~\ref{ass:fdm} holds with $\gamma=\beta$, the score has a finite second moment since $\sigma_t\le\sigma_{\max}$, and the theorems apply with $\rho_m=m^{1/2-\beta}$: the current innovation smooths the score through the scale exactly as it does through the location in the previous example. The empirical median and the empirical quantiles of the training block satisfy Assumption~\ref{ass:train} by ergodicity (Lemma~\ref{lem:ergodicqr} with a constant covariate). Since the law of $Y_t$ is symmetric, the CQR score with the constant oracle endpoints $\mp q_\alpha$ coincides with the centered CMR score, as in the previous example.

\paragraph{Nonlinear recursions.}
Under the contraction condition of Appendix~\ref{app:smoothingproof}, the observations satisfy $\delta^Z_q(j)\le\lambda_H^j\delta^Z_q(0)$, and the scores inherit geometric decay when the oracle endpoints are Lipschitz. Proposition~\ref{prop:sufficient}(a) then gives Assumption~\ref{ass:fdm} with every $\gamma>1/2$ whenever the score has a bounded marginal density, and part (b) applies when the recursion has the additive form $Y_t=h(Z_{t-1})+\varepsilon_t$ with an innovation density bounded by $M_\varepsilon$.

\subsection{Mixing versus Functional Dependence}\label{app:switch}

This subsection compares mixing with the functional dependence measure on two examples of Section~\ref{sec:main}: Gaussian long memory is not strongly mixing, and for the Bernoulli autoregression both the $\beta$-mixing and the switch coefficients equal one, so that the coverage bounds formulated through them are vacuous.

\paragraph{Long memory is not strongly mixing.}\label{app:mixingproof}
A stationary Gaussian sequence is strongly mixing if and only if its spectral density has the form $\abs{P}^2\exp(u+\tilde v)$ with $P$ a trigonometric polynomial, $u$ and $v$ continuous and $\tilde v$ the conjugate function of $v$ \citep[Theorem~7.1]{bradley2005basic}; equivalently, $\log f-\log\abs{P}^2$ has vanishing mean oscillation for some $P$, since the functions of vanishing mean oscillation are exactly the sums $u+\tilde v$ with $u$ and $v$ continuous \citep{sarason1975functions}. Suppose $f(\lambda)=\abs{\lambda}^{2\beta-2}L(\lambda)$ with $1/2<\beta<1$ and $L$ continuous and positive at the origin, as for the linear process with the coefficients of Appendix~\ref{app:conventions}, whose spectral density is proportional to $\abs{2\sin(\lambda/2)}^{2\beta-2}$. Then $\log f-\log\abs{P}^2=-(2-2\beta+2m)\log\abs{\lambda}+\psi(\lambda)$ near the origin, where $m\ge0$ is the multiplicity of the zero of $P$ at the origin and $\psi$ is continuous, hence of vanishing mean oscillation; but $\log\abs{\lambda}$ is not, because its mean oscillation over $[-h,h]$ is the same positive number for every $h>0$, and $2-2\beta+2m>0$. Hence no such representation exists, and the sequence is not strongly mixing, in particular not $\beta$-mixing.

\paragraph{The switch coefficient of the Bernoulli autoregression.}
The coverage guarantees for split conformal prediction on a stationary series are formulated through $\beta$-mixing coefficients \citep{oliveira2024split} or through the switch coefficient of \citet{barber2026predictive}. For a sequence $V=(V_1,\ldots,V_N)$ of scores, a gap $\tau\ge0$ and $1\le k\le N-1-\tau$, the switch coefficient $\Psi_{k,\tau}(V)$ is the total variation distance between the law of $(V_1,\ldots,V_{N-\tau-k},V_{N-k+1},\ldots,V_N)$ and the law of $(V_{\tau+k+1},\ldots,V_N,V_1,\ldots,V_k)$, that is, between the laws of two blocks of lengths $N-\tau-k$ and $k$ separated by the gap $\tau$ in their two temporal orders; for the remaining values of $k$ the two laws coincide by stationarity. The coverage bound of \citet[Theorem~4]{barber2026predictive} subtracts from $1-\alpha$ a gap term, which is at least $\tau$ divided by the number of scores, and the average of $\Psi_{k,\tau}(V)$ over $k$, minimized over the gap; for stationary $\beta$-mixing sequences, the switch coefficient satisfies $\Psi_{k,\tau}(V)\le2\beta(\tau)$ \citep[Proposition~1]{barber2026predictive}. The switch coefficient cannot be bounded through the functional dependence measure in general, and the two routes separate on the first example of Appendix~\ref{app:examples}. Let $W_t=W_{t-1}/2+B_t$ be that process, for which $\delta^Z_q(j)\le2^{-j}$. Its $\beta$-mixing coefficient equals one at every lag, because the conditional law of $W_\tau$ given $\Fcal_0$ is supported on $2^\tau$ points while the marginal law is continuous. Its switch coefficients also equal one. Put $L=N-\tau-k$. Under the first law above, the coordinates at positions $L$ and $L+1$ are $W_L$ and $W_{L+\tau+1}$, and the recursion gives $2^{\tau+1}W_{L+\tau+1}-W_L=2\sum_{j=1}^{\tau+1}2^{j-1}B_{L+j}\in2\mathbb Z$ almost surely. Under the second law, the same coordinates are $W_N$ and $W_1$, and $W_N=2^{-(N-1)}W_1+D_N$ with $D_N$ a function of $B_2,\ldots,B_N$, so that $2^{\tau+1}W_1-W_N\in2\mathbb Z$ requires $W_1$ to lie in a countable set determined by $(B_2,\ldots,B_N)$; since $W_1$ is a function of $(B_s)_{s\le1}$, independent of $(B_2,\ldots,B_N)$, with a continuous law, this event has probability zero. The event $\{2^{\tau+1}v_{L+1}-v_L\in2\mathbb Z\}$ therefore has probability one under the first law and zero under the second, and $\Psi_{k,\tau}(W)=1$ for every $\tau$ and every $1\le k\le N-1-\tau$. The same holds for the CMR score $V_t=\abs{W_t-1}$ of that example, with the event that $2^{\tau+1}w'-w\in2\mathbb Z$ for some $w\in\{1\pm v_L\}$ and $w'\in\{1\pm v_{L+1}\}$, since each of the four sign combinations leads to a relation of the same kind with a nonzero coefficient of $W_1$. Consequently the average switch coefficient is at least $1-(\tau+1)/N$, the coverage bound of \citet{barber2026predictive} is vacuous for every gap once $N>1/\alpha$, and so is the $\beta$-mixing bound of \citet{oliveira2024split}, whereas Theorem~\ref{thm:cmr} with the true center, for which $E_n=0$, gives the coverage error $C_1m^{-1/3}$ with no gap.

\section{Proofs for Section \ref{sec:gaussian}}\label{app:gaussian}

This appendix proves Theorem~\ref{cor:rates} and Proposition~\ref{prop:lower} and bounds the coverage error of the interval. Write $g=\muhat-\mu$ for the error of the center, a function of the training block and of randomness independent of the series, and let $a(g)$ be the $(1-\alpha)$-quantile of $\abs{Z-g}$ for $Z\sim N(0,1)$, the radius that an infinitely long calibration block would return for a center that is off by $g$, so that $a(0)=z_\alpha$. The length of the CMR interval is $2\qR$, where $\qR$ is the $k$-th smallest of the calibration residuals $\abs{Y_{n+i}-\muhat}=\abs{G_{n+i}-g}$. Since $a(g)\ge z_\alpha$ by Lemma~\ref{lem:shiftedradius}(i), the reverse triangle inequality gives
\begin{equation}\label{eq:gaussdecomp}
\big|\abs{\qR-z_\alpha}-\{a(g)-z_\alpha\}\big|\le\abs{\qR-a(g)}.
\end{equation}
The term $a(g)-z_\alpha$ is the effect of the center on the population radius, and Appendix~\ref{app:shifted} shows that it is of second order in $g$. The right-hand side is the error of a finite calibration block that is adjacent to the training block and therefore dependent on the center, and Appendix~\ref{app:conditional} bounds it. The proof of Theorem~\ref{cor:rates} uses \eqref{eq:gaussdecomp} as an upper bound on $\abs{\qR-z_\alpha}$, and the proof of Proposition~\ref{prop:lower} uses it as a lower bound (Appendices~\ref{app:ratesproof} and~\ref{app:lowerproof}). Appendix~\ref{app:hermite} collects the Hermite expansions used in Appendix~\ref{app:conditional} and for the calibration rate $r_m(\beta)$ of Section~\ref{sec:gaussian}, and Appendix~\ref{app:quantilemse} bounds the mean squared error of empirical quantiles of the training block, which gives the rates of Theorem~\ref{cor:rates} for the empirical median and the midpoint. Appendix~\ref{app:coverageproof} uses the bound of Appendix~\ref{app:conditional} on $\abs{\qR-a(g)}$ to show that the coverage error has the order of the length bound of Theorem~\ref{cor:rates}.

Throughout this appendix, \eqref{eq:linearprocess} is in force, $C$ denotes a constant depending only on $\alpha$, $\beta$, $A$ and $a_0$ whose value may change from line to line, and constants that depend on fewer quantities are named. Two elementary consequences of \eqref{eq:linearprocess} are used repeatedly. First, the autocorrelation $\rho(h)=\Cov(G_t,G_{t+h})=\sum_{j\ge0}a_ja_{j+h}$ of the series satisfies, splitting the sum at $j=h$,
\begin{equation}\label{eq:rhobound}
\abs{\rho(h)}\le\sum_{j\ge0}\abs{a_ja_{j+h}}\le C_\rho(1+h)^{1-2\beta},
\end{equation}
where $C_\rho=A^2/(1-\beta)+A^2/(2\beta-1)$; note that $C_\rho\ge1$, because $1=\sum_ja_j^2\le A^2\sum_j(1+j)^{-2\beta}\le A^2\{1+1/(2\beta-1)\}\le C_\rho$. Moreover, the same bound holds for $\Cov(\sum_{j<i}a_j\varepsilon_{t-j},\sum_{j<i'}a_j\varepsilon_{t'-j})$, a covariance of two truncated sums at times $t$ and $t'$, in terms of $h=\abs{t-t'}$. Second, for $s\ge1$,
\begin{equation}\label{eq:sums}
\sum_{h=0}^{s-1}(1+h)^{1-2\beta}\le\frac{s^{2-2\beta}}{2-2\beta},\quad
\frac1s\sum_{h=0}^{s-1}(1+h)^{2-4\beta}\le C_\beta\,r_s(\beta)^2,\quad
\sum_{j\ge i}(1+j)^{-2\beta}\le\frac{i^{1-2\beta}}{2\beta-1}\ (i\ge1),
\end{equation}
where $C_\beta$ equals $1/(3-4\beta)$, $2$ and $1+1/(4\beta-3)$ in the respective regimes $\beta<3/4$, $\beta=3/4$ and $\beta>3/4$, and where $r_s(\beta)$ is the rate \eqref{eq:rm} with $s$ in place of $m$. We also use $s^{1-2\beta}\le r_s(\beta)$ and $1/s\le r_s(\beta)$ for every $s\ge2$; the thresholds $m_0$ and $n_0$ below are at least $2$.

\subsection{Hermite Expansions and the Calibration Rate of the Symmetric Score}\label{app:hermite}

This subsection bounds covariances of functions of a Gaussian series through their Hermite expansions. An even function has no odd Hermite coefficients, so its covariances decay like the squared autocorrelations of the series; this gives the calibration rate $r_m(\beta)$ of the symmetric score (Lemma~\ref{lem:gaussiansigma}), and the same expansion bounds the conditional variance in Appendix~\ref{app:conditional}. The reduction of even functionals of a Gaussian series to Hermite rank two goes back to \citet{breuer1983central} and \citet{dehling1989empirical}. Let $H_k$ denote the probabilists' Hermite polynomials, normalized by $\E H_j(\xi)H_k(\xi)=j!\ind{j=k}$ for $\xi\sim N(0,1)$. For a standard bivariate Gaussian pair $(\xi,\xi')$ with correlation $\varrho$, comparing coefficients in
\begin{equation}\label{eq:hermitegenerating}
\E\,e^{u\xi-u^2/2}e^{v\xi'-v^2/2}=e^{\varrho uv}
\end{equation}
gives $\E H_j(\xi)H_k(\xi')=j!\varrho^j\ind{j=k}$, and therefore, for square-integrable $\psi,\psi'$ with Hermite coefficients $\eta_k=\E\psi(\xi)H_k(\xi)$ and $\eta'_k=\E\psi'(\xi)H_k(\xi)$,
\begin{equation}\label{eq:covexpansion}
\Cov(\psi(\xi),\psi'(\xi'))=\sum_{k\ge1}\frac{\eta_k\eta'_k}{k!}\varrho^k,\qquad \Var(\psi(\xi))=\sum_{k\ge1}\frac{\eta_k^2}{k!}.
\end{equation}

\begin{lemma}\label{lem:evencov}
Let $(\xi,\xi')$ be a standard bivariate Gaussian pair with correlation $\varrho$, and let $\psi$ and $\psi'$ be even and square-integrable. Then $\abs{\Cov(\psi(\xi),\psi'(\xi'))}\le\varrho^2\{\Var(\psi(\xi))\Var(\psi'(\xi'))\}^{1/2}$.
\end{lemma}

\begin{proof}
The odd Hermite coefficients of an even function vanish, so \eqref{eq:covexpansion} runs over even $k\ge2$, where $\abs{\varrho^k}\le\varrho^2$; the Cauchy--Schwarz inequality and the Parseval identity in \eqref{eq:covexpansion} give the bound.
\end{proof}

For indicators of symmetric intervals, whose variances are at most $1/4$, the covariance is at most $\varrho^2/4$, including when the two intervals differ.

\begin{lemma}[Calibration rate of the symmetric score]\label{lem:gaussiansigma}
Let $\Sor_t=\abs{G_t}-z_\alpha$ and $r=z_\alpha/2$, as in Appendix~\ref{app:examples}. Then $\sigma_m\le C_\rho C_\beta^{1/2}\,r_m(\beta)$. If $\rho(h)\ge c_\rho(1+h)^{1-2\beta}$ for all $h\ge h_0$, with $c_\rho>0$, then for every real $a$ the standard deviation of $m^{-1}\sum_{i=1}^m\ind{G_{n+i}\le a}$, the corresponding quantity for a one-sided indicator, is of exact order $m^{1/2-\beta}$.
\end{lemma}

\begin{proof}
For $\abs{s}\le r$ and $a=z_\alpha+s>0$, the indicator $\ind{\Sor_t\le s}=\ind{\abs{G_t}\le a}$ is an even function of $G_t$, so Lemma~\ref{lem:evencov} bounds each covariance $\Cov(\ind{\abs{G_{n+i}}\le a},\ind{\abs{G_{n+j}}\le a})$ by $\rho(i-j)^2/4$, and \eqref{eq:rhobound} and \eqref{eq:sums} give
\[
\Var\{\Fhat(s)\}\le\frac{1}{4m^2}\sum_{i,j\le m}\rho(i-j)^2\le\frac{C_\rho^2}{2m}\sum_{h<m}(1+h)^{2-4\beta},
\]
which is at most $C_\rho^2C_\beta r_m(\beta)^2/2$. For the one-sided indicator, the first Hermite coefficient of $x\mapsto\ind{x\le a}$ is $-\phi(a)\ne0$, and by \eqref{eq:covexpansion} and the Cauchy--Schwarz inequality the remaining terms contribute at most $\rho(h)^2/4$, so the covariance at lag $h$ is $\phi(a)^2\rho(h)+O(\rho(h)^2)$, of exact order $\rho(h)$ for all large $h$; the variance of the partial sum is therefore of order $m^{3-2\beta}$, which gives the scale $m^{1/2-\beta}$.
\end{proof}

\subsection{The Shifted Population Radius}\label{app:shifted}

This subsection deals with the term $a(g)-z_\alpha$ in \eqref{eq:gaussdecomp}, the effect of the center on the population radius. Because the law of $Z$ is symmetric, $a(g)$ is an even function of $g$, so it has no first-order term, and $a(g)-z_\alpha$ is of order $g^2$ instead of $\abs{g}$. Lemma~\ref{lem:shiftedradius} makes this precise: part (iii) gives the two-sided bounds, of order $\min(g^2,1)$ and $g^2$, that enter the proofs of Proposition~\ref{prop:lower} and Theorem~\ref{cor:rates}, and part (iv) gives the exact second-order constant; parts (i) and (ii), the position of $a(g)$ relative to $\abs{g}$ and a lower bound on the density of $\abs{Z-g}$ near $a(g)$, are used in Appendix~\ref{app:conditional}.

For $g\in\R$ and $w\ge0$, let $H_g(w)=\Prob(\abs{Z-g}\le w)=\Phi(w-g)+\Phi(w+g)-1$ with $Z\sim N(0,1)$. For each $g$ the map $w\mapsto H_g(w)$ is continuous and strictly increasing from $0$ to $1$, with derivative $H_g'(w)=\phi(w-g)+\phi(w+g)\le2\phi(0)<1$, so the $(1-\alpha)$-quantile $a(g)$ of $\abs{Z-g}$ is the unique solution of $H_g(a(g))=1-\alpha$, and $a(0)=z_\alpha$. Put
\begin{equation}\label{eq:radiusconstants}
\begin{aligned}
w_*&=\frac{z_\alpha}2,\qquad
f_*=\min\{\phi(u):\Phi^{-1}(1-\alpha)-w_*\le u\le z_\alpha+w_*\},\\
\kappa_\alpha&=\inf_{0<u\le1}\frac{\phi(z_\alpha-u)-\phi(z_\alpha+u)}{u},\qquad
f_1=\min\{\phi(u):\abs{u}\le\max(z_\alpha,1)\}.
\end{aligned}
\end{equation}
The constant $\kappa_\alpha$ is positive because the ratio is continuous and positive on $(0,1]$ with limit $2z_\alpha\phi(z_\alpha)$ at zero, and the ratio is at most $2\sup\abs{\phi'}=2\phi(1)$ by the mean value theorem.

\begin{lemma}[Shifted radius]\label{lem:shiftedradius}
\begin{enumerate}
\item[(i)] $a$ is even and nondecreasing in $\abs{g}$, and $\Phi^{-1}(1-\alpha)\le a(g)-\abs{g}\le z_\alpha$ for every $g$.
\item[(ii)] $H_g'(w)\ge f_*$ for every $g$ and every $w$ with $\abs{w-a(g)}\le w_*$.
\item[(iii)] $\tfrac12\kappa_\alpha\min(g^2,1)\le a(g)-z_\alpha\le C_\alpha g^2$ for every $g$, with $C_\alpha=\max\{\phi(1)/f_1,1\}$.
\item[(iv)] $a(g)=z_\alpha+\tfrac12z_\alpha g^2+O(g^4)$ as $g\to0$.
\end{enumerate}
\end{lemma}

\begin{proof}
(i) $H_g(w)$ is even in $g$, and $\partial_gH_g(w)=\phi(w+g)-\phi(w-g)\le0$ for $g,w\ge0$, so $H_g(w)$ is nonincreasing in $\abs{g}$ and $a(g)$ is nondecreasing in $\abs{g}$; this is Anderson's inequality for the intervals $[g-w,g+w]$. For $g\ge0$, $1-\alpha=H_g(a(g))\le\Phi(a(g)-g)$ gives the lower bound on $a(g)-g$, and $H_g(g+z_\alpha)=\Phi(z_\alpha)+\Phi(2g+z_\alpha)-1\ge2\Phi(z_\alpha)-1=1-\alpha$ gives $a(g)\le g+z_\alpha$; evenness covers $g<0$.

(ii) For $g\ge0$ and $w\ge0$, $\abs{w-g}\le w+g$, so $H_g'(w)\ge\phi(w-g)=\phi(w-\abs{g})$, and the same holds for $g<0$ by evenness. If $\abs{w-a(g)}\le w_*$, then by (i) $w-\abs{g}\in[\Phi^{-1}(1-\alpha)-w_*,z_\alpha+w_*]$, on which $\phi\ge f_*$.

(iii) For $0\le g\le1$,
\[
H_0(z_\alpha)-H_g(z_\alpha)=\int_0^g\{\phi(z_\alpha-u)-\phi(z_\alpha+u)\}\,du\in\big[\kappa_\alpha g^2/2,\ \phi(1)g^2\big].
\]
Since $a(g)\ge z_\alpha$ and $H_g'<1$, $1-\alpha=H_g(a(g))\le H_g(z_\alpha)+a(g)-z_\alpha$, and $1-\alpha=H_0(z_\alpha)$, so $a(g)-z_\alpha\ge H_0(z_\alpha)-H_g(z_\alpha)\ge\kappa_\alpha g^2/2$; for $\abs{g}\ge1$, monotonicity gives $a(g)-z_\alpha\ge a(1)-z_\alpha\ge\kappa_\alpha/2$. For the upper bound with $\abs{g}\le1$, every $u\in[z_\alpha,a(g)]$ satisfies $u-\abs{g}\in[z_\alpha-1,z_\alpha]$ by (i), so $H_g'(u)\ge\phi(u-\abs{g})\ge f_1$ there, and $H_g(a(g))-H_g(z_\alpha)=H_0(z_\alpha)-H_g(z_\alpha)\le\phi(1)g^2$ gives $a(g)-z_\alpha\le\phi(1)g^2/f_1$; for $\abs{g}>1$, (i) gives $a(g)-z_\alpha\le\abs{g}\le g^2$.

(iv) The function $(g,w)\mapsto H_g(w)$ is smooth with $\partial_wH_g(w)>0$, so $a$ is smooth by the implicit function theorem, and even, so $a'(0)=0$. Differentiating $H_g(a(g))=1-\alpha$ twice at $g=0$ gives $\partial_g^2H_0(z_\alpha)+\partial_wH_0(z_\alpha)a''(0)=0$ with $\partial_g^2H_0(z_\alpha)=2\phi'(z_\alpha)=-2z_\alpha\phi(z_\alpha)$ and $\partial_wH_0(z_\alpha)=2\phi(z_\alpha)$, whence $a''(0)=z_\alpha$, and the odd derivatives vanish.
\end{proof}

The second-order behavior of residual empirical processes on the estimation sample under long memory is classical \citep{beran1991slowly,koul2010goodness,chan2008residual,lorek2014empirical}; there it arises from the exact cancellation of the first-order term of the error process by the estimated location, which uses the same observations, whereas here the residuals are out of sample, the object is the expected error of an order statistic at finite $n$ and $m$, and the mechanism is the evenness of the population radius. The mechanism rests on the symmetry of the error law and not on the absolute-value form of the score: for a skewed error the radius $a(g)$ is no longer even, its derivative at $g=0$ is proportional to the difference of the error density at the two oracle endpoints, and the center error re-enters the length at first order.

For symmetric innovations that are not Gaussian, Lemma~\ref{lem:shiftedradius}(i)--(iii) remain valid with the marginal density of the series in place of $\phi$, provided that it is positive, continuously differentiable, nonincreasing on $[0,\infty)$ and has a negative derivative at the radius, with the constants redefined accordingly; the rates, however, use Gaussianity throughout, in the Hermite expansions of Lemmas~\ref{lem:evencov} and \ref{lem:gaussiansigma}, in the Gaussian conditioning of Lemmas~\ref{lem:condmean}--\ref{lem:finitecal}, and in the Gaussian H\"older inequality of Lemma~\ref{lem:quantilemse}.

\subsection{Calibration Residuals Conditional on the Training Innovations}\label{app:conditional}

This subsection deals with the right-hand side of \eqref{eq:gaussdecomp}, the error of a finite calibration block. The difficulty is that the calibration block is adjacent to the training block, so that the center and the calibration residuals are dependent. We condition on the innovations up to time $n$: each calibration observation splits into a part that is determined by these innovations and a Gaussian part that is independent of them, so that, conditionally, the center is fixed and each residual is Gaussian, shifted by the error of the center and by the memory of the training period, which decays with the distance from the training block. Lemmas~\ref{lem:condmean} and~\ref{lem:condvar} bound the conditional mean and variance of the empirical distribution function of the residuals, and Lemma~\ref{lem:finitecal} turns these bounds into a bound on $\E\abs{\qR-a(g)}$ of order $r_m(\beta)$, up to a term that couples the error of the center with the memory of the training period.

Let $\Fcal_n$ denote the sigma-field generated by the innovations up to time $n$ and by the randomness of the training algorithm, so that $g=\muhat-\mu$ is $\Fcal_n$-measurable. For $1\le i\le m+1$ decompose the future observation into its predictable and its innovative parts,
\begin{equation}\label{eq:decomposition}
G_{n+i}=\Pi_i+U_i,\qquad \Pi_i=\sum_{j\ge i}a_j\varepsilon_{n+i-j},\qquad U_i=\sum_{j=0}^{i-1}a_j\varepsilon_{n+i-j},
\end{equation}
so that $\Pi_i$ is $\Fcal_n$-measurable and $(U_1,\ldots,U_{m+1})$ is a centered Gaussian vector independent of $\Fcal_n$, with $s_i^2=\Var(U_i)=\sum_{j<i}a_j^2\in[a_0^2,1]$ and $\tau_i=\E \Pi_i^2=\sum_{j\ge i}a_j^2=1-s_i^2\le A^2i^{1-2\beta}/(2\beta-1)$ by \eqref{eq:sums}. Put
\begin{equation}\label{eq:pastquantities}
T_1=\frac1m\sum_{i=1}^m\abs{\Pi_i},\quad \Pi_2=\frac1m\sum_{i=1}^m\Pi_i^2,\quad \bar\tau=\frac1m\sum_{i=1}^m\tau_i\le C_\tau m^{1-2\beta},\quad C_\tau=\frac{A^2}{(2\beta-1)(2-2\beta)},
\end{equation}
where the bound on $\bar\tau$ follows from \eqref{eq:sums}; note that $\E \Pi_2=\bar\tau$ and $\E T_1^2\le\bar\tau$ by Jensen's inequality. By \eqref{eq:rhobound} applied to the truncated sums, $\abs{\Cov(U_i,U_j)}\le C_\rho(1+\abs{i-j})^{1-2\beta}$, so the correlations $R_{ij}=\Cov(U_i,U_j)/(s_is_j)$ satisfy $\abs{R_{ij}}\le C_\rho a_0^{-2}(1+\abs{i-j})^{1-2\beta}$, and by \eqref{eq:sums}
\begin{equation}\label{eq:Rsums}
\max_{i\le m}\sum_{j=1}^m\abs{R_{ij}}\le\frac{C_\rho}{a_0^2}\,\frac{m^{2-2\beta}}{1-\beta},\qquad
\frac1{m^2}\sum_{i,j=1}^mR_{ij}^2\le\frac{2C_\rho^2C_\beta}{a_0^4}\,r_m(\beta)^2.
\end{equation}
Finally, for $w\ge0$, $b\in\R$ and $s\in[\abs{a_0},1]$, let $\pi_w(b,s)=\Prob(\abs{s\xi+b}\le w)=\Phi\{(w-b)/s\}-\Phi\{(-w-b)/s\}$ for $\xi\sim N(0,1)$; then $\pi_w(-g,1)=H_g(w)$, and
\begin{equation}\label{eq:piderivatives}
\partial_b\pi_w(0,s)=0,\qquad\abs{\partial_b^2\pi_w(b,s)}\le\frac{2\phi(1)}{a_0^2},\qquad\abs{\partial_s\pi_w(b,s)}\le\frac{2\phi(1)}{\abs{a_0}},
\end{equation}
because $\partial_b\pi_w(b,s)=s^{-1}[\phi\{(w+b)/s\}-\phi\{(w-b)/s\}]$, $\abs{\phi'}\le\phi(1)$ and $\abs{x\phi(x)}\le\phi(1)$.

\begin{lemma}[Conditional mean]\label{lem:condmean}
Let $\Hhat(w)=m^{-1}\sum_{i=1}^m\ind{\abs{G_{n+i}-g}\le w}$ for $w\ge0$, with $g$ $\Fcal_n$-measurable. Then, for every $\Fcal_n$-measurable $w\ge0$,
\[
\big|\E\{\Hhat(w)\mid\Fcal_n\}-H_g(w)\big|\le B,\qquad B=K_1\big(\abs{g}T_1+\Pi_2+\bar\tau\big),\qquad K_1=\frac{2\phi(1)}{\min(a_0^2,\abs{a_0})},
\]
and $\E B\le K_1\{(\bar\tau\,\E g^2)^{1/2}+2\bar\tau\}$.
\end{lemma}

\begin{proof}
Conditionally on $\Fcal_n$, $\abs{G_{n+i}-g}=\abs{U_i+\Pi_i-g}$ with $U_i\sim N(0,s_i^2)$, so $\E\{\ind{\abs{G_{n+i}-g}\le w}\mid\Fcal_n\}=\pi_w(\Pi_i-g,s_i)$. A Taylor expansion in the first argument about $-g$ and \eqref{eq:piderivatives} give $\abs{\pi_w(\Pi_i-g,s_i)-\pi_w(-g,s_i)}\le\abs{\partial_b\pi_w(-g,s_i)}\abs{\Pi_i}+\phi(1)\Pi_i^2/a_0^2\le2\phi(1)a_0^{-2}(\abs{g}\abs{\Pi_i}+\Pi_i^2/2)$, since $\partial_b\pi_w(\cdot,s)$ vanishes at zero and is $2\phi(1)a_0^{-2}$-Lipschitz. The mean value theorem in the second argument gives $\abs{\pi_w(-g,s_i)-\pi_w(-g,1)}\le2\phi(1)\abs{a_0}^{-1}(1-s_i)\le2\phi(1)\abs{a_0}^{-1}\tau_i$, because $1-s_i\le1-s_i^2=\tau_i$. Averaging over $i$ proves the first claim, and the Cauchy--Schwarz inequality with $\E T_1^2\le\bar\tau$ and $\E \Pi_2=\bar\tau$ proves the second.
\end{proof}

\begin{lemma}[Conditional variance]\label{lem:condvar}
In the setting of Lemma~\ref{lem:condmean}, for every $\Fcal_n$-measurable $w\ge0$,
\[
\Var\{\Hhat(w)\mid\Fcal_n\}\le W,\qquad W=K_2\big\{r_m(\beta)^2+m^{1-2\beta}(\Pi_2+g^2)\big\},
\]
where $K_2=\max\{8\phi(1)^2C_\rho/(a_0^4(1-\beta)),\ C_\rho^2C_\beta/(2a_0^4)\}$.
\end{lemma}

\begin{proof}
Write $\xi_i=U_i/s_i$ and $b_i=\Pi_i-g$, so that, conditionally on $\Fcal_n$, the summand $\ind{\abs{G_{n+i}-g}\le w}=\psi_i(\xi_i)$ with $\psi_i(x)=\ind{\abs{s_ix+b_i}\le w}$ and $(\xi_i)$ a standard Gaussian vector with correlations $R_{ij}$. The first Hermite coefficient of $\psi_i$ is $\eta_{i,1}=\E[\xi\ind{(-w-b_i)/s_i\le\xi\le(w-b_i)/s_i}]=\phi\{(-w-b_i)/s_i\}-\phi\{(w-b_i)/s_i\}$, which vanishes at $b_i=0$ and is $2\phi(1)\abs{a_0}^{-1}$-Lipschitz in $b_i$, so $\abs{\eta_{i,1}}\le2\phi(1)\abs{a_0}^{-1}\abs{b_i}$. By \eqref{eq:covexpansion}, the Cauchy--Schwarz inequality over $k\ge2$ and $\Var(\psi_i(\xi_i))\le1/4$,
\[
\abs{\Cov(\psi_i(\xi_i),\psi_j(\xi_j)\mid\Fcal_n)}\le\frac{4\phi(1)^2}{a_0^2}\abs{b_ib_j}\abs{R_{ij}}+\frac{R_{ij}^2}4,
\]
also on the diagonal. Since $\sum_{i,j}\abs{b_ib_j}\abs{R_{ij}}\le(\max_i\sum_j\abs{R_{ij}})\sum_ib_i^2$ for a symmetric matrix with nonnegative entries, and $\sum_ib_i^2\le2m(\Pi_2+g^2)$, the bounds \eqref{eq:Rsums} give
\[
\Var\{\Hhat(w)\mid\Fcal_n\}\le\frac{4\phi(1)^2}{a_0^2}\cdot\frac{C_\rho m^{2-2\beta}}{a_0^2(1-\beta)}\cdot\frac{2m(\Pi_2+g^2)}{m^2}+\frac{C_\rho^2C_\beta}{2a_0^4}r_m(\beta)^2,
\]
which is the claim.
\end{proof}

\begin{lemma}[Finite-calibration comparison]\label{lem:finitecal}
Let $g$ be $\Fcal_n$-measurable with $\E g^2<\infty$, and let $\qR$ be the $k$-th smallest of $\abs{G_{n+i}-g}$, $1\le i\le m$. There are $m_0$, depending only on $\alpha$, and $K_3$, depending only on $\alpha$, $\beta$, $A$ and $a_0$, such that for $m\ge m_0$,
\[
\E\abs{\qR-a(g)}\le K_3\big\{r_m(\beta)+m^{1-2\beta}+(m^{1-2\beta}\E g^2)^{1/2}\big\}.
\]
\end{lemma}

\begin{proof}
Let $m_0=\max\{8/(f_*w_*),2(1-\alpha)/\alpha\}$, so that $d_m=k/m-(1-\alpha)<2/m\le f_*w_*/4$ and $\min(k,m-k+1)\ge m\min(1-\alpha,\alpha/2)$ for $m\ge m_0$, by the argument given for \eqref{eq:m0}.

\emph{Step 1: a conditional tail bound.} Let $0<t\le w_*$ with $f_*t>B+d_m$. The event $\qR>a(g)+t$ implies $\Hhat(a(g)+t)<k/m\le1-\alpha+d_m$, whereas Lemma~\ref{lem:condmean} and Lemma~\ref{lem:shiftedradius}(ii) give $\E\{\Hhat(a(g)+t)\mid\Fcal_n\}\ge H_g(a(g)+t)-B\ge1-\alpha+f_*t-B$; the event therefore requires a downward deviation of $\Hhat(a(g)+t)$ from its conditional mean by more than $f_*t-B-d_m>0$. The event $\qR<a(g)-t$ implies $\Hhat(a(g)-t)\ge k/m\ge1-\alpha$, whereas the conditional mean is at most $1-\alpha-f_*t+B$, so it requires an upward deviation of at least $f_*t-B\ge f_*t-B-d_m$. Chebyshev's inequality conditionally on $\Fcal_n$ and Lemma~\ref{lem:condvar} give
\begin{equation}\label{eq:condtail}
\Prob\{\abs{\qR-a(g)}>t\mid\Fcal_n\}\le\min\Big\{1,\frac{2W}{(f_*t-B-d_m)^2}\Big\}.
\end{equation}

\emph{Step 2: a conditional second moment.} By Lemma~\ref{lem:shiftedradius}(i), $\abs{g}-a(g)\in[-z_\alpha,-\Phi^{-1}(1-\alpha)]$, so $\big|\abs{x-g}-a(g)\big|\le\abs{x}+C_z$ for every real $x$, with $C_z=z_\alpha+\abs{\Phi^{-1}(1-\alpha)}$. The quantity $\qR-a(g)$ is the $k$-th order statistic of $T_i=\abs{G_{n+i}-g}-a(g)$, and for any real numbers, $T_{(k)}^2\le\sum_iT_i^2/\min(k,m-k+1)$, because at least $m-k+1$ of the $T_i$ are at least $T_{(k)}$ when $T_{(k)}\ge0$ and at least $k$ of them are at most $T_{(k)}$ when $T_{(k)}<0$. Since $\E(G_{n+i}^2\mid\Fcal_n)=s_i^2+\Pi_i^2\le1+\Pi_i^2$,
\begin{equation}\label{eq:condsecond}
M^2:=\E\{(\qR-a(g))^2\mid\Fcal_n\}\le\frac{2\sum_{i\le m}(1+\Pi_i^2+C_z^2)}{m\min(1-\alpha,\alpha/2)}\le C_M(1+\Pi_2),
\end{equation}
with $C_M=(2+2C_z^2)/\min(1-\alpha,\alpha/2)$.

\emph{Step 3: the conditional expectation.} Put $a'=(B+d_m)/f_*$ and $b'=(2W)^{1/2}/f_*$, so that the right-hand side of \eqref{eq:condtail} is $\min\{1,b'^2/(t-a')^2\}$. If $a'<w_*$, integrating \eqref{eq:condtail} over $0<t\le w_*$ as in the proof of Proposition~\ref{prop:quantileform} gives $\E\{\min(\abs{\qR-a(g)},w_*)\mid\Fcal_n\}\le a'+2b'$, and the conditional Cauchy--Schwarz inequality gives $\E\{(\abs{\qR-a(g)}-w_*)_+\mid\Fcal_n\}\le M\min\{1,b'/(w_*-a')\}$. Let $b_*=f_*w_*/4$. On the event $\{B\le b_*\}$ we have $a'\le w_*/2$, hence
\begin{equation}\label{eq:condexp}
\E\{\abs{\qR-a(g)}\mid\Fcal_n\}\le a'+2b'+\frac{2Mb'}{w_*}\qquad\text{on }\{B\le b_*\}.
\end{equation}

\emph{Step 4: the unconditional expectation.} Taking expectations in \eqref{eq:condexp} and using $\E b'\le(2\E W)^{1/2}/f_*$ and $\E(Mb')\le(\E M^2\,\E b'^2)^{1/2}$,
\[
\E\big[\abs{\qR-a(g)}\ind{B\le b_*}\big]\le\frac{\E B+d_m}{f_*}+\frac{2(2\E W)^{1/2}}{f_*}\Big\{1+\frac{(\E M^2)^{1/2}}{w_*}\Big\}.
\]
On the complementary event, $\E\{\abs{\qR-a(g)}\mid\Fcal_n\}\le M\le C_M^{1/2}(1+\Pi_2^{1/2})$ by \eqref{eq:condsecond}, so that, by Markov's inequality and the Cauchy--Schwarz inequality,
\[
\E\big[\abs{\qR-a(g)}\ind{B>b_*}\big]\le C_M^{1/2}\big\{\Prob(B>b_*)+(\bar\tau\Prob(B>b_*))^{1/2}\big\}\le C_M^{1/2}\Big\{\frac{3\E B}{2b_*}+\frac{\bar\tau}2\Big\}.
\]
Now $\E M^2\le2C_M$ because $\bar\tau\le1$, $\E W\le K_2\{r_m(\beta)^2+m^{1-2\beta}(\bar\tau+\E g^2)\}$ by Lemma~\ref{lem:condvar}, $\E B\le K_1\{(\bar\tau\E g^2)^{1/2}+2\bar\tau\}$ by Lemma~\ref{lem:condmean}, $\bar\tau\le C_\tau m^{1-2\beta}$, $d_m\le2/m\le2r_m(\beta)$, and $(m^{1-2\beta}\bar\tau)^{1/2}\le C_\tau^{1/2}m^{1-2\beta}$. Collecting the terms proves the lemma.
\end{proof}

\subsection{Empirical Quantiles of a Long-Memory Gaussian Sample}\label{app:quantilemse}

The bound of Theorem~\ref{cor:rates} involves the center only through its mean squared error. For the sample mean this is the variance of $\Gbar$, which is bounded directly in Appendix~\ref{app:ratesproof}; for the empirical median and for the midpoint of two empirical quantiles it requires the mean squared error of an empirical quantile of the training block, which this subsection bounds by a multiple of $n^{1-2\beta}$ for every $n$. The same bound gives the explicit rate for Assumption~\ref{ass:train} mentioned after that assumption.

\begin{lemma}\label{lem:quantilemse}
Let $\tau\in(0,1)$, $\zeta_\tau=\Phi^{-1}(\tau)$ and $j_n=\lceil n\tau\rceil$. Under \eqref{eq:linearprocess}, $\E(G_{(j_n)}-\zeta_\tau)^2\le C_qn^{1-2\beta}$ for all $n\ge1$, where $G_{(j_n)}$ is the $j_n$-th smallest of $G_1,\ldots,G_n$ and $C_q$ depends only on $\tau$, $\beta$ and $A$.
\end{lemma}

\begin{proof}
Let $\Gamma_n$ be the correlation matrix of $(G_1,\ldots,G_n)$ and $\lambda_n$ its largest eigenvalue, so that $1\le\lambda_n\le\max_i\sum_j\abs{\rho(i-j)}\le1+2C_\rho n^{2-2\beta}/(2-2\beta)\le C_\lambda n^{2-2\beta}$ by \eqref{eq:rhobound}--\eqref{eq:sums}, with $C_\lambda=1+C_\rho/(1-\beta)$. Fix $u\in\R$ and $t>0$, and let $f_i(x)=\exp[t\{\ind{x\le u}-\Phi(u)\}]$. The Gaussian H\"older inequality of \citet[Theorem~1]{chen2015improved} states that for a centered Gaussian vector with covariance matrix $T$, nonnegative measurable functions $f_i$ and exponents $p_i>0$ such that $T\preceq\mathrm{diag}(p_1T_{11},\ldots,p_nT_{nn})$, one has $\E\prod_if_i(X_i)\le\prod_i\{\E f_i(X_i)^{p_i}\}^{1/p_i}$; for two standard Gaussian variables with correlation $\varrho$ and $p_1=p_2=p$ the condition reads $(p-1)^2\ge\varrho^2$, which is Nelson's hypercontractivity condition. Since $\Gamma_n$ has unit diagonal and $\Gamma_n\preceq\lambda_nI$, the condition holds with all exponents equal to $\lambda_n$, and the inequality gives
\[
\E\prod_{i=1}^nf_i(G_i)\le\prod_{i=1}^n\big\{\E f_i(G_i)^{\lambda_n}\big\}^{1/\lambda_n}\le\exp\Big(\frac{n\lambda_nt^2}8\Big),
\]
where the last step is Hoeffding's lemma for the centered variable $\ind{G_i\le u}-\Phi(u)$, whose range has length one. With $F_n(u)=n^{-1}\sum_{i\le n}\ind{G_i\le u}$, the Chernoff bound with $t=4x/\lambda_n$, applied also to $-\{\ind{x\le u}-\Phi(u)\}$, yields
\begin{equation}\label{eq:trainingecdf}
\Prob\{\abs{F_n(u)-\Phi(u)}\ge x\}\le2\exp(-2nx^2/\lambda_n),\qquad x>0.
\end{equation}
Let $f_\tau=\min\{\phi(v):\abs{v-\zeta_\tau}\le1\}$, so that $\Phi(\zeta_\tau+t)-\tau\ge f_\tau t$ and $\tau-\Phi(\zeta_\tau-t)\ge f_\tau t$ for $0\le t\le1$, and note that $\tau\le j_n/n\le\tau+1/n$. For $2/(nf_\tau)\le t\le1$, the event $G_{(j_n)}>\zeta_\tau+t$ implies $F_n(\zeta_\tau+t)<j_n/n$, hence $F_n(\zeta_\tau+t)-\Phi(\zeta_\tau+t)<1/n-f_\tau t\le-f_\tau t/2$, and the event $G_{(j_n)}<\zeta_\tau-t$ implies $F_n(\zeta_\tau-t)\ge\tau$, hence $F_n(\zeta_\tau-t)-\Phi(\zeta_\tau-t)\ge f_\tau t$; by \eqref{eq:trainingecdf},
\[
\Prob(\abs{G_{(j_n)}-\zeta_\tau}>t)\le4\exp\Big(-\frac{nf_\tau^2t^2}{2\lambda_n}\Big),\qquad \frac{2}{nf_\tau}\le t\le1.
\]
Integrating $2t$ times the tail probability, the interval $t<2/(nf_\tau)$ contributes at most $4/(nf_\tau)^2$, the interval up to $t=1$ contributes at most $\int_0^\infty8t\exp\{-nf_\tau^2t^2/(2\lambda_n)\}\,dt=8\lambda_n/(nf_\tau^2)$, and the remainder is $\E[(G_{(j_n)}-\zeta_\tau)^2\ind{\abs{G_{(j_n)}-\zeta_\tau}>1}]\le\{\E(G_{(j_n)}-\zeta_\tau)^4\}^{1/2}\,2\exp\{-nf_\tau^2/(4\lambda_n)\}$. The rank inequality of Step 2 in the proof of Lemma~\ref{lem:finitecal}, applied to fourth powers, gives $G_{(j_n)}^4\le\sum_iG_i^4/\min(j_n,n-j_n+1)\le\sum_iG_i^4/\{n\min(\tau,1-\tau)\}$, so $\E(G_{(j_n)}-\zeta_\tau)^4\le8\{3/\min(\tau,1-\tau)+\zeta_\tau^4\}$. Since $\lambda_n/n\le C_\lambda n^{1-2\beta}$, $1/n^2\le n^{1-2\beta}$ and $\exp(-cn^{2\beta-1})\le n^{1-2\beta}/(ce)$ for $c>0$, all three contributions are at most a constant multiple of $n^{1-2\beta}$; for the finitely many $n<2/f_\tau$ the bound follows from the fourth-moment bound by enlarging the constant.
\end{proof}

\subsection{Proof of Theorem \ref{cor:rates}}\label{app:ratesproof}

\begin{proof}[Proof of Theorem \ref{cor:rates}]
\emph{Step 1: any center.} Write $g=\muhat-\mu$, which is $\Fcal_n$-measurable, and assume $\E g^2<\infty$, since otherwise there is nothing to prove. The calibrated radius $\qR$ is the $k$-th smallest of $\abs{Y_{n+i}-\muhat}=\abs{G_{n+i}-g}$, so Lemma~\ref{lem:finitecal} applies, and since $m^{1-2\beta}\le r_m(\beta)$ it gives $\E\abs{\qR-a(g)}\le2K_3\{r_m(\beta)+(m^{1-2\beta}\E g^2)^{1/2}\}$. Lemma~\ref{lem:shiftedradius}(iii) gives $0\le a(g)-z_\alpha\le C_\alpha g^2$, and \eqref{eq:gaussdecomp} gives
\begin{align*}
\E\abs{\len\ChatM-2z_\alpha}=2\E\abs{\qR-z_\alpha}&\le2\E\abs{\qR-a(g)}+2\E\{a(g)-z_\alpha\}\\
&\le4K_3\big\{r_m(\beta)+(m^{1-2\beta}\E g^2)^{1/2}\big\}+2C_\alpha\E g^2.
\end{align*}
Finally $(m^{1-2\beta}\E g^2)^{1/2}\le\{m^{1-2\beta}+\E g^2\}/2\le\{r_m(\beta)+\E g^2\}/2$, which proves \eqref{eq:secondorder}. The center enters the concentration of $\qR$ around $a(g)$ only through $m^{1-2\beta}\E g^2$: a shift $g$ breaks the evenness of the indicator $\ind{\abs{G_{n+i}-g}\le w}$, whose first Hermite coefficient is of order $g$ and carries the long-memory variance $m^{1-2\beta}$ (Lemma~\ref{lem:condvar}). The center is never treated as independent of the calibration block: the argument conditions on the training innovations and retains the conditional means $\Pi_i-g$ of the calibration residuals, whose interaction with the center appears in the term $(m^{1-2\beta}\E g^2)^{1/2}$.

\emph{Step 2: the three centers.} The sample mean is a function of $Y_1,\ldots,Y_n$ with $\muhat-\mu=\Gbar$, and
\[
v_n=\Var(\Gbar)=\frac1{n^2}\sum_{s,t=1}^n\rho(s-t)\le\frac1n\Big\{1+2\sum_{h=1}^{n-1}\abs{\rho(h)}\Big\}\le\frac{2C_\rho}{n}\,\frac{n^{2-2\beta}}{2-2\beta}=\frac{C_\rho}{1-\beta}\,n^{1-2\beta}
\]
by \eqref{eq:rhobound}--\eqref{eq:sums} and $C_\rho\ge1$. The empirical median is the $\lceil n/2\rceil$-th smallest of $Y_1,\ldots,Y_n$, so that $\muhat-\mu=G_{(\lceil n/2\rceil)}$, and Lemma~\ref{lem:quantilemse} with $\tau=1/2$ gives $\E(\muhat-\mu)^2\le C_qn^{1-2\beta}$. For the midpoint $\widehat c=(\lhat+\uhat)/2$ of $\lhat=Y_{(\lceil n\alpha/2\rceil)}=\mu+G_{(\lceil n\alpha/2\rceil)}$ and $\uhat=Y_{(\lceil n(1-\alpha/2)\rceil)}=\mu+G_{(\lceil n(1-\alpha/2)\rceil)}$, since $\zeta_{\alpha/2}+\zeta_{1-\alpha/2}=0$ for the normal quantiles $\zeta_\tau=\Phi^{-1}(\tau)$,
\[
\widehat c-\mu=\tfrac12\big\{\big(G_{(\lceil n\alpha/2\rceil)}-\zeta_{\alpha/2}\big)+\big(G_{(\lceil n(1-\alpha/2)\rceil)}-\zeta_{1-\alpha/2}\big)\big\},
\]
and Lemma~\ref{lem:quantilemse} at the two levels gives $\E(\widehat c-\mu)^2\le C_qn^{1-2\beta}$, with $C_q$ the larger of the two constants. Step 1 with these three bounds proves the second claim. By the identity of Section~\ref{sec:gaussian}, which holds for any real endpoints, also when $\lhat>\uhat$, the CQR interval with the empirical quantiles as fitted endpoints is the CMR interval with center $\widehat c$. The same lemma gives $\E(\EQ)^2\le2C_qn^{1-2\beta}$ for the training error $\EQ=\max\{\abs{\lhat-\ell_\alpha},\abs{\uhat-u_\alpha}\}$ of these endpoints, which is the explicit rate mentioned after Assumption~\ref{ass:train}.
\end{proof}

The identity used in Step 2 relies on the constant covariate, and the improvement concerns an error of location. With a covariate, a width error that varies with the covariate enters the length at first order: for CQR with $\mu=0$, fitted endpoints $\lhat(x)=-(z_\alpha+ex)$ and $\uhat(x)=z_\alpha+ex$ at a covariate $X_t=\pm1$ independent of $G_t$, the common correction is of order $e^2$ and the length error at each covariate value is $2e+O(e^2)$; a width error common to all covariates cancels in the calibration, by the same identity.

\subsection{Proof of Proposition \ref{prop:lower}}\label{app:lowerproof}

\begin{proof}[Proof of Proposition \ref{prop:lower}]
\emph{Step 1: any center.} Write $g=\muhat-\mu$ and assume $\E g^2<\infty$, since otherwise the right-hand side of \eqref{eq:lowerrisk} is $-\infty$. By \eqref{eq:gaussdecomp}, Lemma~\ref{lem:shiftedradius}(iii) and Lemma~\ref{lem:finitecal},
\begin{align*}
\E\abs{\len\ChatM-2z_\alpha}&=2\E\abs{\qR-z_\alpha}\ge2\E\{a(g)-z_\alpha\}-2\E\abs{\qR-a(g)}\\
&\ge\kappa_\alpha\E\min(g^2,1)-2K_3\big\{r_m(\beta)+m^{1-2\beta}+(m^{1-2\beta}\E g^2)^{1/2}\big\},
\end{align*}
which is \eqref{eq:lowerrisk} with $c_1=\kappa_\alpha$ and $C=4K_3$, since $m^{1-2\beta}\le r_m(\beta)$.

\emph{Step 2: a common Gaussian component of the center.} Let $V=n^{-1/2}\sum_{s=-2n+1}^{-n}\varepsilon_s\sim N(0,1)$, a function of the innovations preceding the training block, and let $b_t=\Cov(G_t,V)=n^{-1/2}\sum_{s=-2n+1}^{-n}a_{t-s}$ for $1\le t\le n$. The indices $t-s$ range over $[t+n,t+2n-1]\subset[n+1,3n-1]$, so for $n\ge j_0$ the lower bound on the coefficients gives
\[
b_t\ge n^{-1/2}\cdot n\cdot c\,(3n)^{-\beta}=h_n,\qquad h_n=c\,3^{-\beta}n^{1/2-\beta}.
\]
Write $W_t=G_t-b_tV$. The vector $(W_1,\ldots,W_n)$ is jointly Gaussian with $V$ and uncorrelated with it, hence independent of $V$. Since $T$ is translation equivariant, $g=T(Y_1,\ldots,Y_n)-\mu=T(G_1,\ldots,G_n)=T(W+bV)$, where $b=(b_1,\ldots,b_n)^\tr$. For fixed $W$, put $f_W(v)=T(W+bv)$; for $h\ge0$, every coordinate of $W+b(v+h)$ exceeds the corresponding coordinate of $W+bv+h_nh\mathbf1$, so monotonicity and equivariance give $f_W(v+h)\ge f_W(v)+h_nh$. If $f_W(0)\ge0$, then $g\ge h_n$ on the event $\{V\ge1\}$; if $f_W(0)<0$, then $g\le f_W(0)-h_n\abs{V}\le-h_n$ on the event $\{V\le-1\}$. Since $V$ is independent of $W$ and $\Prob(V\ge1)=\Prob(V\le-1)=\Phi(-1)$,
\[
\Prob(\abs{g}\ge h_n)\ge\Phi(-1),\qquad \E\min(g^2,1)\ge\Phi(-1)\min(h_n^2,1)=\Phi(-1)c^23^{-2\beta}n^{1-2\beta}
\]
for $n\ge n_0$, where $n_0\ge\max(j_0,2)$ is such that $h_n\le1$. No Gaussianity or unbiasedness of $g$ is used.

\emph{Step 3: a calibration block large relative to the training block.} Let $\nu=2\beta-1\in(0,1)$ and $c_2=\Phi(-1)c^23^{-2\beta}$. By Steps 1 and 2 and $\E g^2\le C_gn^{-\nu}$, for $n\ge n_0$ and $m\ge m_0$,
\[
\E\abs{\len\ChatM-2z_\alpha}\ge c_1c_2n^{-\nu}-C\big\{r_m(\beta)+C_g^{1/2}(nm)^{-\nu/2}\big\}.
\]
Let $m\ge Kn^{\max(1,2\nu)}\log(n+1)$ with $K\ge m_0$, and recall that $n\ge2$, so that $\log(n+1)\ge1$. If $\nu<1/2$, then $m\ge Kn$, so $r_m(\beta)=m^{-\nu}\le K^{-\nu}n^{-\nu}$ and $(nm)^{-\nu/2}\le K^{-\nu/2}n^{-\nu}$. If $\nu=1/2$, then, since $\log(m+1)/m$ is decreasing in $m$ and $Kn\log(n+1)+1\le K(n+1)^2$, $r_m(\beta)^2\le\{\log K+2\log(n+1)\}/\{Kn\log(n+1)\}\le(\log K+2)/(Kn)$, and $(nm)^{-1/4}\le K^{-1/4}n^{-1/2}$. If $\nu>1/2$, then $m\ge Kn^{2\nu}$, so $r_m(\beta)=m^{-1/2}\le K^{-1/2}n^{-\nu}$ and $(nm)^{-\nu/2}\le K^{-\nu/2}n^{-\nu(1+2\nu)/2}\le K^{-\nu/2}n^{-\nu}$, since $2\nu>1$. In each regime $r_m(\beta)+C_g^{1/2}(nm)^{-\nu/2}\le\epsilon_Kn^{-\nu}$ with $\epsilon_K\to0$ as $K\to\infty$, so for $K$ large the right-hand side of the last display is at least $c_1c_2n^{-\nu}/2$ and $r_m(\beta)\le n^{-\nu}$, which gives \eqref{eq:lowerbound} with $c_3=c_1c_2/4$.
\end{proof}

The sample mean, the empirical median and the midpoint of the empirical $\alpha/2$- and $(1-\alpha/2)$-quantiles are nondecreasing in each coordinate and translation equivariant, and Step 2 of the proof of Theorem~\ref{cor:rates} bounds their mean squared error by a multiple of $n^{1-2\beta}$, so they satisfy the conditions of Proposition~\ref{prop:lower} on the center.

\subsection{Coverage under Gaussian Long Memory}\label{app:coverageproof}

This subsection bounds the coverage error of the interval of Theorem~\ref{cor:rates}. The bound has the order of the length bound \eqref{eq:secondorder}, so the interval attains the nominal level at the rate of its length, although its center converges at the slower rate $n^{1/2-\beta}$.

\begin{corollary}[Coverage]\label{cor:coverage}
Assume \eqref{eq:linearprocess}, and let $\muhat$ be as in Theorem~\ref{cor:rates}, with $\E(\muhat-\mu)^2<\infty$. For all $n$ and all $m\ge m_0$, with $m_0$ as in Theorem~\ref{cor:rates},
\begin{equation}\label{eq:learnedcoverage}
\begin{aligned}
\big|\Prob\{Y_{n+m+1}\in\ChatM\}-(1-\alpha)\big|&\le C\big[r_m(\beta)+\{m^{1-2\beta}\E(\muhat-\mu)^2\}^{1/2}+m^{1-2\beta}\big]\\
&\le C\big\{\E(\muhat-\mu)^2+r_m(\beta)\big\}.
\end{aligned}
\end{equation}
For the sample mean, the empirical median and the midpoint of Theorem~\ref{cor:rates}, the coverage error is therefore at most $C\{n^{-(2\beta-1)}+r_m(\beta)\}$.
\end{corollary}

The middle term of \eqref{eq:learnedcoverage} is the interaction between the error of the center and the part of the test observation that the training innovations predict, to which a shifted center is sensitive at first order.

\begin{proof}[Proof of Corollary~\ref{cor:coverage}]
Write $g=\muhat-\mu$. The calibrated radius $\qR$ is the $k$-th smallest of $\abs{G_{n+i}-g}$, $1\le i\le m$, and $Y_{n+m+1}\in\ChatM$ exactly when $\abs{G_{n+m+1}-g}\le\qR$. Let $\Fcal_{n+m}$ denote the sigma-field generated by the innovations up to time $n+m$ and by the randomness of the training algorithm. Conditionally on $\Fcal_{n+m}$, $G_{n+m+1}=a_0\varepsilon_{n+m+1}+\Pi'$ with $\Pi'=\sum_{j\ge1}a_j\varepsilon_{n+m+1-j}$ measurable with respect to $\Fcal_{n+m}$, so $G_{n+m+1}$ is Gaussian with variance $a_0^2$, whereas $g$, $a(g)$ and $\qR$ are $\Fcal_{n+m}$-measurable. Hence the conditional probability $\Prob(\abs{G_{n+m+1}-g}\le q\mid\Fcal_{n+m})=\Phi\{(q+g-\Pi')/\abs{a_0}\}-\Phi\{(-q+g-\Pi')/\abs{a_0}\}$ is Lipschitz in $q$ with constant $2\phi(0)/\abs{a_0}$, and
\[
\big|\Prob\{\abs{G_{n+m+1}-g}\le\qR\}-\Prob\{\abs{G_{n+m+1}-g}\le a(g)\}\big|\le\frac{2\phi(0)}{\abs{a_0}}\,\E\abs{\qR-a(g)}.
\]
For the second probability, condition on $\Fcal_n$ and use the decomposition \eqref{eq:decomposition} at $i=m+1$: $\Prob\{\abs{G_{n+m+1}-g}\le a(g)\mid\Fcal_n\}=\pi_{a(g)}(\Pi_{m+1}-g,s_{m+1})$, whereas $\pi_{a(g)}(-g,1)=H_g(a(g))=1-\alpha$ by the definition of $a(g)$. The bound on a single summand in the proof of Lemma~\ref{lem:condmean}, at $i=m+1$, and the Cauchy--Schwarz inequality give
\begin{align*}
\big|\Prob\{\abs{G_{n+m+1}-g}\le a(g)\}-(1-\alpha)\big|&\le K_1\E\big(\abs{g\Pi_{m+1}}+\Pi_{m+1}^2+\tau_{m+1}\big)\\
&\le K_1\big\{(\tau_{m+1}\E g^2)^{1/2}+2\tau_{m+1}\big\},
\end{align*}
and $\tau_{m+1}\le A^2m^{1-2\beta}/(2\beta-1)$ by \eqref{eq:sums}. Combining the two displays with Lemma~\ref{lem:finitecal} proves the first inequality in \eqref{eq:learnedcoverage}; the second follows from $m^{1-2\beta}\le r_m(\beta)$ and $2(xy)^{1/2}\le x+y$. For the three centers, Step~2 of the proof of Theorem~\ref{cor:rates} gives $\E g^2\le Cn^{1-2\beta}$.
\end{proof}

\section{Experimental Design and Complete Results}\label{app:experiments}

This appendix specifies the studies of Section~\ref{sec:experiments} and reports their complete results. Appendix~\ref{app:conventions} gives the conventions shared by all studies and the Gaussian linear process, Appendices~\ref{app:study1} and~\ref{app:study2} give Studies~1 and~2, Appendix~\ref{app:extension} extends Study~2 to non-Gaussian innovations, and Appendix~\ref{app:study3} gives Study~3.

\subsection{Conventions and Simulation}\label{app:conventions}

This subsection fixes the conventions of all studies and the simulation of the Gaussian linear process.

All studies use $\alpha=.1$ and the rank $k=\lceil(m+1)(1-\alpha)\rceil$, without interpolation. Each trajectory is one stationary path, whose first $n$ observations form the training block and whose next $m$ observations form the calibration block. A cell, that is, a model with fixed block sizes, consists of $1000$ independent trajectories, each generated from its own random stream, a PCG64 generator seeded through NumPy's SeedSequence with a fixed base seed and the identifiers of the cell and the trajectory. A reported value is the average over the $1000$ trajectories, and its Monte Carlo standard error is the sample standard deviation divided by $1000^{1/2}$. The studies were run in Python~3.12.3 with NumPy~2.5.3 and SciPy~1.18.1.

The Gaussian linear process of Studies~1--3 is \eqref{eq:linearprocess} with the coefficients $a_0=c$ and $a_j=a_{j-1}(1-\beta/j)$ for $j\ge1$, where the constant $c>0$ makes $\sum_ja_j^2=1$. These coefficients are positive and asymptotically proportional to $j^{-\beta}$, so that \eqref{eq:linearprocess} holds, and so does the coefficient lower bound of Proposition~\ref{prop:lower}. The autocorrelation $\rho(h)=\sum_{j\ge0}a_ja_{j+h}$ satisfies $\rho(0)=1$ and $\rho(h)=\rho(h-1)(h-\beta)/(h-1+\beta)$ for $h\ge1$, and it is asymptotically proportional to $h^{1-2\beta}$; the variance of the mean of $n$ consecutive observations, $v_n=n^{-2}\{n+2\sum_{j=1}^{n-1}(n-j)\rho(j)\}$, is computed exactly. A path of length $N$ is generated exactly by circulant embedding of its covariance matrix, with an embedding whose size is the smallest power of two that is at least $2(N-1)$, so that the path has the stationary law, without truncation of \eqref{eq:linearprocess} and without burn-in.

\subsection{Study 1: Four Dependence Setups}\label{app:study1}

This subsection specifies Study~1 and reports its complete results for CQR and CMR.

For every integer $t$, let $A_t$, $\xi_t$ and $B_t$ be independent, with $A_t$ uniform on $[-1,1]$, $\xi_t$ standard normal and $B_t$ Bernoulli$(1/2)$, and \iid\ over $t$. The covariate is $X_t=A_t$ and the response is $Y_t=f(X_t)+s(X_t)h(U_t)$, with $f$, $s$ and $h$ as in Section~\ref{sec:experiments}, where $U_t$ is uniform on $[0,1]$ in each of the four setups:
\begin{align*}
&\text{(1) independent: } U_t=\Phi(\xi_t);\qquad
\text{(2) AR(1): } U_t=\Phi(H_t),\quad H_t=.9H_{t-1}+.19^{1/2}\xi_t;\\
&\text{(3) Bernoulli: } U_t=(U_{t-1}+B_t)/2;\qquad
\text{(4) long-memory linear process: } U_t=\Phi(G_t),
\end{align*}
where $G_t=\sum_{j\ge0}a_j\xi_{t-j}$ in (4), with the coefficients of Appendix~\ref{app:conventions} and $\beta=.8$. The process in (3) is $W_t/2$ for the Bernoulli autoregression $W_t$ of Section~\ref{sec:main} and is simulated exactly on a $128$-bit binary state; (2) and (3) start from their stationary laws. Setups (3) and (4) are not strongly mixing (Appendices~\ref{app:examples} and~\ref{app:switch}), and neither is $(X_t,Y_t)$, of which $U_t$ is a function.

The learner is linear quantile regression on the basis $(1,x,x^2)$ over the coefficient box $[-5,5]^3$, solved exactly as a linear program, at the levels $.05$ and $.95$ for CQR and $.5$ for CMR. Since $h$ is increasing, the conditional $\tau$-quantile of $Y_t$ given $X_t=x$ is $f(x)+s(x)h(\tau)$, so the learner is correctly specified; the oracle CQR interval has the endpoints $f(x)+s(x)h(.05)$ and $f(x)+s(x)h(.95)$, and the oracle CMR interval is $\mu(x)\pm q_\alpha$, with $\mu(x)=f(x)+s(x)h(.5)$ and $q_\alpha=.752$ the $(1-\alpha)$-quantile of $\abs{Y_t-\mu(X_t)}$. Assumptions~\ref{ass:train}--\ref{ass:fdm} hold in all four setups: Assumption~\ref{ass:train} by Lemma~\ref{lem:ergodicqr}, since every setup is ergodic and the learner is correctly specified; Assumption~\ref{ass:local} because the conditional density of $Y_t$ given $X_t$ lies between $1/2.6$ and $1$ on its support (Appendix~\ref{app:margin}); and Assumption~\ref{ass:fdm} by Proposition~\ref{prop:sufficient}, through part (a) with every $\gamma>1/2$ in setups (1)--(3), whose scores have geometrically decaying functional dependence, and through part (b) with $\gamma=\beta=.8$ in setup (4). Hence $\rho_m=m^{-1/2}$ in setups (1)--(3) and $\rho_m=m^{-.3}$ in setup (4).

The marginal coverage is estimated by averaging over the trajectories the conditional coverage probability of the calibrated interval given the observed past, computed from the known conditional law of $U_{n+m+1}$; this has a smaller Monte Carlo variance than averaging coverage indicators. The length error is $\E\sup_x\abs{\len\ChatQ(x)-\len\CQ(x)}$ for CQR and $2\E\abs{\qR-q_\alpha}$ for CMR, and the cells are $n=m\in\{128,\ldots,4096\}$ for each setup.

\begin{figure}[t]
\begin{center}
\includegraphics[width=.9\linewidth]{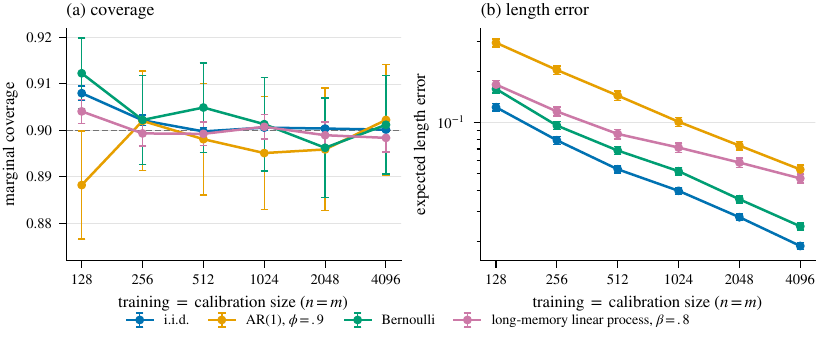}
\end{center}
\caption{Study~1, CMR with quadratic median regression, on the trajectories of Figure~\ref{fig:g1}. (a) Marginal coverage; the dashed line is the nominal level. (b) Length error $2\E\abs{\qR-q_\alpha}$, on a logarithmic scale. Both panels show pointwise 95\% Monte Carlo intervals.}
\label{fig:g1cmr}
\end{figure}

Figure~\ref{fig:g1} shows the CQR results and Figure~\ref{fig:g1cmr} the CMR results. The coverage of both procedures lies between $.888$ and $.912$ in all cells; in setup (1), where the scores are exchangeable, it agrees with the exact value $k/(m+1)$, for example $117/129=.907$ at $m=128$. The length errors decrease in all four setups. Least-squares fits of the logarithm of the length error on $\log n$ over the four largest sizes give the decay exponents $.53$, $.53$, $.51$ and $.51$ for CQR and $.50$, $.48$, $.50$ and $.29$ for CMR, in the order of the setups. In setup (4), over the tested sizes, the length error of CMR decreases like $m^{-.3}$, the rate $\rho_m$ of Theorem~\ref{thm:cmr}, whereas that of CQR decreases like $m^{-1/2}$, as in the other setups.

The oracle calibration errors show the same contrast. On the calibration blocks of setup (4), these errors, $\E\abs{\SQ_{(k)}}$ and $\E\abs{\SM_{(k)}-q_\alpha}$, which use the oracle scores and involve no training, decrease with the exponents $.50$ and $.30$. At its population threshold, the CQR oracle indicator $\ind{\SQ_t\le0}=\ind{\abs{G_t}\le z_\alpha}$ is an even function of $G_t$, like the indicators of Section~\ref{sec:gaussian}, so its first Hermite coefficient vanishes and its covariances decay like $\rho(h)^2$ (Lemma~\ref{lem:evencov}), which leads to the rate $m^{-1/2}$ at $\beta=.8$. The CMR indicator is not even in $G_t$, because $h$ is not linear, and its nonzero first Hermite coefficient leads to the rate $m^{1/2-\beta}=m^{-.3}$. This explains the oracle calibration rates; the fitted CQR length can also retain a first-order training-error term, so the observed slope does not establish its asymptotic rate.

\subsection{Study 2: The Sharper Length Rate}\label{app:study2}

This subsection specifies Study~2 and the three numbers of Table~\ref{tab:rates}.

The observations are $Y_t=\mu+G_t$, with $\mu=0$, which is no loss of generality because the procedure is translation equivariant, and with $G_t$ the Gaussian linear process of Appendix~\ref{app:conventions} for $\beta\in\{.6,.75,.9\}$. On each path of length $n+m$, the center is the training sample mean $\muhat=\bar Y_n$, the calibrated radius $\qR$ is the $k$-th smallest of $\abs{Y_{n+i}-\muhat}$, and the known-center radius $\qRor$ is the $k$-th smallest of $\abs{Y_{n+i}-\mu}$. The cells are $n=m\in\{128,\ldots,4096\}$. Figure~\ref{fig:l1} shows the three quantities of the study: the mean squared error $\E(\muhat-\mu)^2$ of the center, which equals $v_n$; the oracle calibration error $\E\abs{\qRor-z_\alpha}$; and the length error $\E\abs{\len\ChatM-2z_\alpha}=2\E\abs{\qR-z_\alpha}$.

The three numbers of Table~\ref{tab:rates} summarize the rates of these quantities, each computed over the five largest sizes $n=m\in\{256,\ldots,4096\}$. The estimation error exponent checks the rate of the center: it is minus the least-squares slope, with an intercept, of the logarithm of the mean squared error on $\log n$, to be compared with $2\beta-1$. The oracle calibration error ratio checks whether the oracle calibration error decreases at the rate $r_m(\beta)$ of \eqref{eq:rm}: it is the least-squares slope of the logarithm of this error on $\log r_m(\beta)$, which equals the ratio of the decay exponent of the error to that of $r_m(\beta)$ when both are powers of $m$; the value one means that the error decreases at the rate $r_m(\beta)$, and a larger value that it decreases faster. The total length error ratio is the same number for the length error, on $\log r_n(\beta)$, and checks whether the length error decreases at the rate of Theorem~\ref{cor:rates}. The Monte Carlo standard error of each number is its standard deviation over $2000$ bootstrap replications that resample the $1000$ trajectories of each cell with replacement, keeping the quantities of a trajectory together.

\begin{table}[t]
\caption{Study~2 and its extension, rates over $n=m\in\{256,\ldots,4096\}$, with Monte Carlo standard errors in parentheses; the reference is $2\beta-1$ for the estimation error exponent and one for the two ratios. The center is the training mean for Gaussian innovations and the training median for the non-Gaussian innovations of Appendix~\ref{app:extension}.}
\label{tab:rates}
\begin{center}
\footnotesize
\begin{tabular}{lccccc}
\toprule
Innovations & $\beta$ & $2\beta-1$ & \begin{tabular}[b]{@{}c@{}}Estimation error\\ exponent\end{tabular} & \begin{tabular}[b]{@{}c@{}}Oracle calibration\\ error ratio\end{tabular} & \begin{tabular}[b]{@{}c@{}}Total length\\ error ratio\end{tabular} \\
\midrule
Gaussian & $.60$ & $.20$ & .211 (.021) & 1.029 (.055) & .986 (.046) \\
 & $.75$ & $.50$ & .544 (.019) & 1.118 (.028) & 1.110 (.028) \\
 & $.90$ & $.80$ & .806 (.021) & 1.056 (.022) & 1.080 (.022) \\
\addlinespace[2pt]
Laplace & $.60$ & $.20$ & .202 (.021) & .905 (.055) & 1.043 (.045) \\
 & $.75$ & $.50$ & .459 (.020) & 1.095 (.026) & 1.106 (.027) \\
 & $.90$ & $.80$ & .829 (.020) & 1.054 (.022) & 1.057 (.023) \\
\addlinespace[2pt]
centered exponential & $.60$ & $.20$ & .206 (.019) & .973 (.056) & 1.026 (.044) \\
 & $.75$ & $.50$ & .487 (.021) & .720 (.026) & .686 (.026) \\
 & $.90$ & $.80$ & .802 (.020) & .891 (.022) & .876 (.022) \\
\bottomrule

\end{tabular}
\end{center}
\end{table}

The Gaussian rows of Table~\ref{tab:rates} report the three numbers. The estimated mean squared errors agree with the exact $v_n$ within $1.7$ Monte Carlo standard errors in all $18$ cells, and the estimation error exponents $.211$, $.544$ and $.806$ are close to $2\beta-1$: the center converges slowly. The two ratios are close to one. At $\beta=.6$, they correspond to decay exponents of about $.21$ and $.20$, against $.1$ for the rate $m^{1/2-\beta}$ of Theorem~\ref{thm:cmr}, and at $\beta=.9$ to about $.53$ and $.54$, against $.5$ for $r_m(\beta)$ and $.4$ for $m^{1/2-\beta}$. At $\beta=.75$ the ratios $1.118$ and $1.110$ exceed one: over this range the errors decrease faster than $\{\log(m+1)/m\}^{1/2}$. The oracle calibration error and the length error therefore follow the rates of Theorem~\ref{cor:rates} and decrease faster than the center converges.

\subsection{Extension of Study 2: Non-Gaussian Innovations}\label{app:extension}

Section~\ref{sec:gaussian} proves the sharper rate for Gaussian innovations and attributes it to the symmetry of their law. This subsection repeats Study~2 with two non-Gaussian innovation laws and the same coefficients. The Laplace law is symmetric, and tests whether the improvement extends beyond Gaussian innovations, as Section~\ref{sec:gaussian} expects. The centered exponential law is asymmetric by design, and tests whether the improvement rests on symmetry, in which case the length error should approach the rate of Theorem~\ref{thm:cmr}.

The observations are $Y_t=G_t=\sum_{j\ge0}a_j\varepsilon_{t-j}$, with the coefficients of Appendix~\ref{app:conventions}, $\beta\in\{.6,.75,.9\}$ and \iid\ innovations with mean zero and variance one: Laplace with scale $2^{-1/2}$, or $\mathrm{Exp}(1)-1$, which has skewness $2$. The center is the empirical median of the training block, which estimates the center $c$ of the CMR oracle interval $c\pm q$, the median of $G_t$; the radius $q$ is the $(1-\alpha)$-quantile of $\abs{G_t-c}$, computed by numerical inversion of the characteristic function of $G_t$ with an error below $2\times10^{-6}$. The cells are $n=m\in\{128,\ldots,4096\}$, and the quantities and the three numbers are those of Study~2, with the known center $c$ in the oracle calibration error. Paths with non-Gaussian innovations cannot be generated by circulant embedding. For a path of length $N$ and $L$ the smallest power of two that is at least $\max(2^{20},64N)$, we compute $G_t=H_t+R_t$ for $1\le t\le N$, where $H_t=\sum_{s=1-L}^{t}a_{t-s}\varepsilon_s$ is obtained exactly from $L+N$ simulated innovations by a fast Fourier convolution, and the remote part $R_t=\sum_{s\le-L}a_{t-s}\varepsilon_s$ is replaced by a Gaussian vector with its exact covariance. This keeps the covariance of the path and changes only the higher cumulants of $R_t$, a weighted sum of infinitely many innovations with small weights. With Gaussian innovations, the same algorithm reproduces Study~2 within Monte Carlo error.

\begin{figure}[t]
\begin{center}
\includegraphics[width=.8\linewidth]{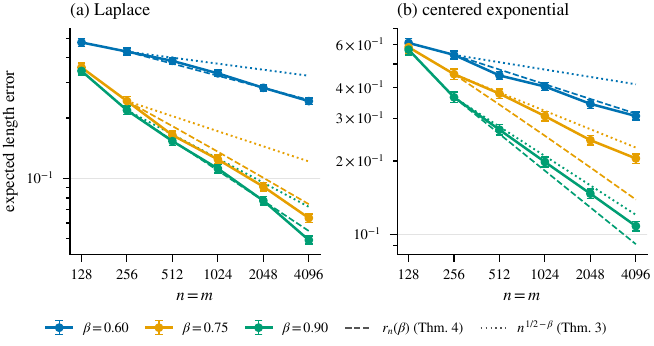}
\end{center}
\caption{Extension of Study~2, CMR with the empirical median as center, for (a) Laplace and (b) centered exponential innovations and $\beta\in\{.6,.75,.9\}$: length error with $n=m$, with the rate $r_n(\beta)$ of Theorem~\ref{cor:rates} (dashed) and the rate $n^{1/2-\beta}$ of Theorem~\ref{thm:cmr} (dotted), matched to the empirical curves at $256$. Pointwise 95\% Monte Carlo intervals.}
\label{fig:ext}
\end{figure}

Figure~\ref{fig:ext} shows the length error, and Table~\ref{tab:rates} gives the three numbers. With Laplace innovations, the findings of Study~2 hold: the estimation error exponents are close to $2\beta-1$, and the oracle calibration and total length error ratios are close to one, between $.905$ and $1.106$. With centered exponential innovations, the total length error ratios at $\beta=.75$ and $.9$ are $.686$ and $.876$, well below one, whereas the same ratios computed with the rate $n^{1/2-\beta}$ of Theorem~\ref{thm:cmr} in place of $r_n(\beta)$ are $1.171$ and $1.095$, close to one: the length error follows the rate of Theorem~\ref{thm:cmr} more closely than that of Theorem~\ref{cor:rates}. At $\beta=.6$, however, the total length error ratio is $1.026$, and $2.051$ with the rate of Theorem~\ref{thm:cmr}.

The exception at $\beta=.6$ has a simple source: long memory makes the law of $G_t$ more symmetric. Under the centered exponential law, the skewness of $G_t$ is $2\sum_{j\ge0}a_j^3$, which equals $.76$, $1.60$ and $1.95$ at $\beta=.6$, $.75$ and $.9$, since $a_0=.70$, $.92$ and $.99$: the stronger the memory, the more the coefficients spread each innovation over many lags, and the closer the law of $G_t$ is to a Gaussian law. The asymmetry enters the length through the first-order term of the population radius. Let $f$ be the density of $G_t$, $c$ its median and $q$ the oracle radius, and let $a(g)$ now denote the $(1-\alpha)$-quantile of $\abs{G_t-c-g}$, so that $a(0)=q$; implicit differentiation gives
\[
a'(0)=\frac{f(c-q)-f(c+q)}{f(c-q)+f(c+q)},
\]
which vanishes for a symmetric law, as in Appendix~\ref{app:shifted}. The densities $f(c-q)$ and $f(c+q)$ at the oracle endpoints are $.078$ and $.106$, $.005$ and $.108$, and below $10^{-4}$ and $.101$ at $\beta=.6$, $.75$ and $.9$, so that $\abs{a'(0)}=.15$, $.91$ and $1.00$. By the reduction principle for the empirical process of long-memory linear processes \citep{wu2003empirical}, the empirical distribution function of the calibration block differs from the distribution function of $G_t$ at $x$, to first order, by $-f(x)\bar G_m$, where $\bar G_m=m^{-1}\sum_{i=1}^mG_{n+i}$. The known-center radius, the $k$-th smallest of $\abs{G_{n+i}-c}$, therefore has the first-order error $-a'(0)\bar G_m$, whose expected absolute value is about $\abs{a'(0)}(2v_m/\pi)^{1/2}$.

At $m=4096$, this first-order part is about $.05$, $.09$ and $.03$ at $\beta=.6$, $.75$ and $.9$, against observed oracle calibration errors of $.16$, $.10$ and $.05$. At $\beta=.75$ and $.9$ it is the larger part of the error, which then decreases at about the rate $m^{1/2-\beta}$. At $\beta=.6$ the error is dominated by the part that Gaussian innovations also produce, about $.14$ at $m=4096$, which decreases like $r_m(\beta)=m^{-.2}$, whereas the first-order part decreases only like $m^{-.1}$ and would dominate only for $m$ of the order of $10^8$. Accordingly, the total length error ratios computed with the rate of Theorem~\ref{thm:cmr}, $2.05$, $1.17$ and $1.09$ at $\beta=.6$, $.75$ and $.9$, decrease as the skewness grows. Hence the more symmetric the law of the series, the closer the length error is to the rate of Theorem~\ref{cor:rates}, and the more skewed, the closer it is to the rate of Theorem~\ref{thm:cmr}. The extension thus supports both the expectation of Section~\ref{sec:gaussian} for symmetric innovations and its dependence on symmetry.

\subsection{Study 3: The Lower Bound}\label{app:study3}

This subsection specifies Study~3 and reports its complete results. The study illustrates the training term of the lower bound \eqref{eq:lowerbound} of Proposition~\ref{prop:lower}: for a fixed training block, the length error remains at least of the order of the mean squared error of the center, however long the calibration block.

The model is that of Study~2 with $\beta=.9$ and the sample mean as center. Proposition~\ref{prop:lower} gives the lower bound when $m\ge Kn^{\max(1,4\beta-2)}\log(n+1)$, which is $m\ge Kn^{1.6}\log(n+1)$ at $\beta=.9$, for a constant $K$ that the proposition does not specify. The cells are therefore $n\in\{64,128,256,512,1024\}$ and $m=2^{\lceil\log_2(Kn^{1.6})\rceil}$ for $K\in\{1,4,16,64\}$, from $1024$ to $4{,}194{,}304$. The grid omits the factor $\log(n+1)$, which varies only by a factor $1.66$ over the tested $n$, less than the factor $4$ between consecutive values of $K$.

The quantity of interest is the length error $\E\abs{\len\ChatM-2z_\alpha}=2\E\abs{\qR-z_\alpha}$. To separate the contribution of the center from that of the calibration, let $g=\Gbar$ be the error of the training mean and $a(g)$ the radius of Appendix~\ref{app:gaussian}, the $(1-\alpha)$-quantile of $\abs{Z-g}$ for $Z\sim N(0,1)$. Taking expectations in \eqref{eq:gaussdecomp} gives
\[
\big|\,2\E\abs{\qR-z_\alpha}-2\E\{a(g)-z_\alpha\}\big|\le2\E\abs{\qR-a(g)}.
\]
The term $2\E\{a(g)-z_\alpha\}$, which we call the reference, is the length error of the interval calibrated on an infinitely long block, and it depends only on the center. The right-hand side, which we call the calibration error, is the error of calibrating on a block of length $m$. The display thus states that the length error differs from the reference by at most the calibration error. Since $g\sim N(0,v_n)$ exactly, the reference equals $4\int_0^\infty\{a(v_n^{1/2}u)-z_\alpha\}\phi(u)\,du$ and is computed by quadrature; by Lemma~\ref{lem:shiftedradius}(iv) it is asymptotically $z_\alpha v_n$, and its values $.054$, $.031$, $.018$, $.010$ and $.006$ for $n=64,\ldots,1024$ are within $2.2\%$ of $z_\alpha v_n$. The calibration error is estimated on each trajectory, with $a(g)$ solved for the realized $g$.

Figure~\ref{fig:s3} reports all $20$ cells. As $K$ grows, the calibration error halves from each value of $K$ to the next, as the calibration rate $m^{-1/2}$ predicts, for example from $.078$ at $K=1$ to $.040$, $.020$ and $.010$ at $n=64$, and the length error approaches the reference at every training size. With $K=64$, the calibration error is at most $.20$ times the reference, so that the length error lies within about $20\%$ of the reference; it is in fact between $.997$ and $1.080$ times the reference, and from $K=16$ to $K=64$ it changes by less than $1.96$ standard errors of the change at every size. With $K=64$, the length error decreases from $.056$ at $n=64$ to $.0065$ at $n=1024$, a factor of $8.57$, against the factor $9.19$ of $n^{-.8}$ over the same range. Once the calibration is accurate, the length error therefore stays at the level $z_\alpha v_n$ set by the center, which decreases only as the training block grows, at the rate $n^{-(2\beta-1)}$ of the mean squared error of the center.

\section{Application: Day-Ahead Load Forecast Errors}\label{app:application}

Every morning, the New York Independent System Operator (NYISO) publishes a forecast of the hourly electricity load of the New York control area (NYCA) for the following day \citep{nyiso2026loadforecast}, and it later publishes the load that was actually served \citep{nyiso2026actualload}. A prediction interval around the forecast states how far the actual load may fall from it, and such probabilistic load forecasts have become important to the planning and operation of energy systems \citep{hong2016probabilistic}. We construct the intervals by split conformal calibration of the percentage forecast error $Y=100(A-F)/F$, where $A$ is the actual load and $F$ the forecast, with one daily series for each hour of the day. The errors are persistent and, as we show below, have long memory, so exchangeability fails, whereas the theory of Section~\ref{sec:main} in principle allows for long memory. We use data from 2005 on, because the median error shifts by about $2.5$ percentage points in January 2005. The years 2005--2013 served to choose and fix the design, which was then applied once to the $111{,}010$ hourly targets from 1 January 2014 to 31 August 2026, to measure how often the $90\%$ intervals contain the actual load and how wide they are.

For the target day $D$, the calibration block consists of the $m$ days ending on $D-2$, all of whose hours have ended before the forecast for $D$ is posted, and the training block of the $n$ days immediately before it, with $n=m\in\{365,730,1461\}$. The target is thus two days after the calibration block, the fixed horizon $h=2$ of Appendix~\ref{app:jointproof}. The origin rolls daily, and every fit is recomputed at every origin. Missing days are skipped, and $k=\lceil(m'+1)(1-\alpha)\rceil$ with $\alpha=.1$ and $m'$ the number of available calibration days.
\begin{itemize}
\item CMR uses a constant covariate. The center $\muhat$ is the training median of $Y$, the radius $\ShatM_{(k)}$ is the $k$-th smallest value of $\abs{Y_t-\muhat}$ over the calibration block, and the interval for the load is $[F\{1+(\muhat-\ShatM_{(k)})/100\},\ F\{1+(\muhat+\ShatM_{(k)})/100\}]$.
\item CQR fits linear quantile regressions of $Y$ at the levels $.05$ and $.95$ on the forecast, standardized on the training block, and on two annual harmonics of the day of the year $d$, $\cos(2\pi jd/365.25)$ and $\sin(2\pi jd/365.25)$ for $j=1,2$, and calibrates them with the score of Section~\ref{sec:cqr}. Each fit is a linear program, accepted only if a feasible dual solution certifies its objective to within $10^{-6}$; no fit failed.
\item The width of the hindsight interval, defined after the evaluation, is a proxy for the oracle width. With a constant covariate, the oracle CMR interval of Section~\ref{sec:cmr} is $\mu\pm q_\alpha$, with $\mu$ the median of $Y$ under the stationary law and $q_\alpha$ the $(1-\alpha)$-quantile of $\abs{Y-\mu}$, and its width is $2q_\alpha$. CMR estimates $\mu$ from the $n$ training errors and $q_\alpha$ from the $m$ calibration errors. The hindsight interval $\tilde\mu\pm\tilde q$ of an hour estimates both from all $N$ errors of that hour in the evaluation period ($N\ge4613$): $\tilde\mu$ is their median and $\tilde q$ the $\lceil(N+1)(1-\alpha)\rceil$-th smallest value of $\abs{Y-\tilde\mu}$. Under the model of Section~\ref{sec:setup}, the errors are ergodic (Appendix~\ref{app:training}), so $\tilde\mu$ and $\tilde q$ converge to $\mu$ and $q_\alpha$ as $N$ grows when $\mu$ is unique and Assumption~\ref{ass:local} holds; for finite $N$, the hindsight width $2\tilde q$ is an estimate of $2q_\alpha$ with an error of its own. The hindsight interval uses the targets themselves, so it cannot be used for prediction, and it covers $.900$ of them by construction.
\end{itemize}

\begin{figure}[t]
\begin{center}
\includegraphics[width=.9\linewidth]{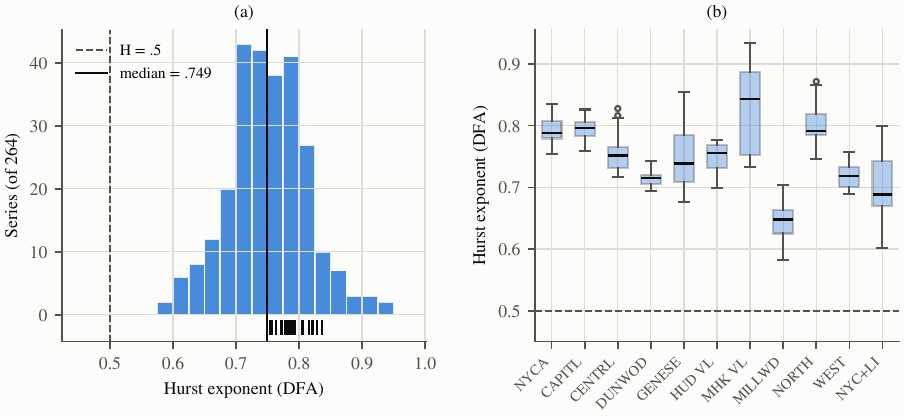}
\end{center}
\caption{Application, Hurst exponents: (a) histogram; (b) box plots of the estimates by area.}
\label{fig:loadhurst}
\end{figure}

Hurst exponents are estimated by detrended fluctuation analysis \citep{peng1994mosaic} for the $264$ hourly error series of NYCA and its ten zone areas (the eleven NYISO load zones, with New York City and Long Island combined) over the evaluation period. Figure~\ref{fig:loadhurst} shows their histogram, with the $24$ NYCA series marked below the axis, and their box plots by area, $24$ hourly series each: all $264$ estimates exceed $.5$, with median $.749$, so the errors have long memory.

\begin{figure}[t]
\begin{center}
\includegraphics[width=.95\linewidth]{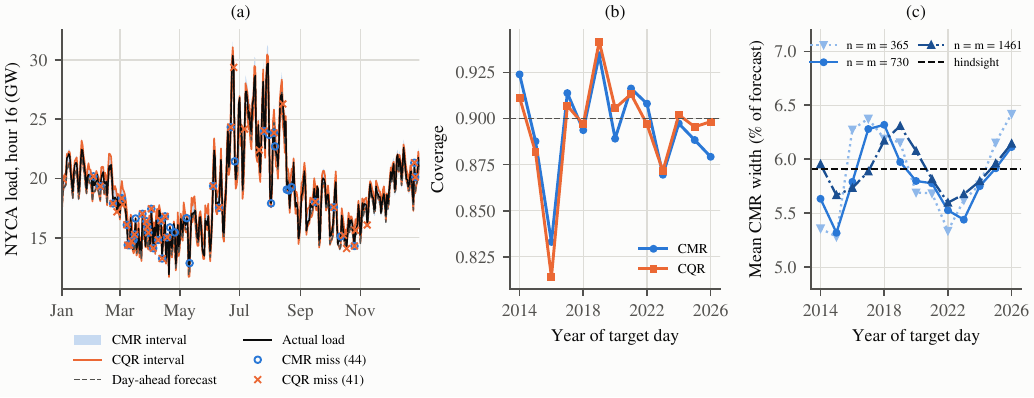}
\end{center}
\caption{Application, NYCA, evaluation period 2014--2026. (a) Hour beginning 16:00, target days in 2025, $n=m=730$: actual load, forecast, CMR interval (band) and CQR interval (lines), in GW, with the misses of each method marked. (b) Coverage by year of CMR and CQR, $n=m=730$, over all 24 hours; the dashed line is the nominal level; 2026 covers January to August. (c) Mean width of CMR by year, in percent of the forecast, for $n=m=365$, $730$ and $1461$; the dashed line is the mean width of the hindsight interval.}
\label{fig:loadmain}
\end{figure}

Figure~\ref{fig:loadmain} shows the results. At $n=m=730$, CMR and CQR both cover $.895$ of the $111{,}010$ targets, $.005$ below the nominal level, and with $n=m=365$ and $1461$ their coverage lies between $.886$ and $.899$. Coverage varies over the years, from $.833$ in 2016 to $.934$ in 2019 for CMR. The mean width of CMR is $5.86\%$, $5.81\%$ and $5.90\%$ of the forecast for $n=m=365$, $730$ and $1461$, or $1017$~MW at $730$, and that of CQR $6.01\%$, $5.86\%$ and $5.96\%$. The mean width of the hindsight interval is $5.91\%$, or $1036$~MW (Figure~\ref{fig:loadmain}(c)), and $5.82\%$ and $5.79\%$ when it is computed from 2014--2019 and from 2020--2026 only. On average, and relative to the mean hindsight width, the CMR width at a single origin with $n=m=730$ differs from the hindsight width of its hour by $9\%$, and the hindsight widths of an hour computed from 2014--2019 and from 2020--2026 differ by $13\%$.

The results make two points. First, although the errors have long memory, both procedures cover close to the nominal level of $90\%$ over more than twelve years of targets that were not used to design them. Second, and more important, the mean width of CMR is within $2\%$ of the hindsight width at every block size: to the accuracy of this proxy, the calibrated interval is on average as wide as the oracle interval. Since the hindsight width is itself an estimate, the comparison concerns average widths, not the width at a single origin.

\end{document}